\documentclass[review]{scrartcl}
\usepackage[utf8]{inputenc}
\usepackage{fontenc}
\usepackage[english]{babel}
\usepackage{csquotes}
\usepackage{amsmath,amssymb,mathtools,amsthm}
\usepackage{graphicx}
\usepackage{mathrsfs}
\usepackage{subcaption}
\usepackage{xargs}
\usepackage{algorithm}
\usepackage[noend]{algpseudocode}
\usepackage{booktabs}
\usepackage{enumitem}
\usepackage{listings}
\usepackage{environ,ifthen,xcolor}
\usepackage{hyperref}
\usepackage{animate}
\usepackage{yhmath}
\usepackage{bbm}
\mathtoolsset{showonlyrefs}
\definecolor{tuklblue}{RGB}{0,95,140}

\setlist[enumerate,1]{label={\upshape{\roman*)}},itemsep=.25\baselineskip,topsep=.25\baselineskip}
\setlist[enumerate,2]{label={\upshape{\alph*)}},itemsep=.25\baselineskip,topsep=.25\baselineskip}
\newcommand{\dx}{\mathrm{d}}
\newcommand{\tT}{\mathrm{T}}

\newcommandx{\abs}[2][1=\@empty]{#1\lvert #2 #1\rvert}
\newcommandx{\norm}[3][1=\@empty,3=\@empty]{#1\lVert #2 #1\rVert_{#3}}
\renewcommand{\vec}[1]{\mathbf{#1}}

\newcommand{\grid}{\mathcal{G}}

\newcommand{\NN}{\mathbb{N}}

\DeclareMathOperator*{\argmin}{arg\,min} %handles subscripts like \lim
\DeclareMathOperator{\prox}{prox} %handle normal subscripts

\DeclareMathOperator{\trace}{trace}

\DeclareMathOperator{\diffeo}{\mathscr{A}}

\DeclareMathOperator{\TV}{TV}

\newtheorem{definition}{Definition}[subsection]
\newtheorem{theorem}[definition]{Theorem}
\newtheorem{lemma}[definition]{Lemma}
\newtheorem{corollary}[definition]{Corollary}

\newtheorem{remark}[definition]{Remark}

\newboolean{useExternalization}
\setboolean{useExternalization}{true}
\newcommand{\externalFolder}{imgpdf/}
\newcommand{\externalOnly}[1]{
	\ifthenelse{\boolean{useExternalization}}{#1}{}%
}
\NewEnviron{TikZOrPDF}[2][]{% Optional first argument is set so {}
	\ifthenelse{\boolean{useExternalization}}{%
		\tikzsetnextfilename{#2}%
		\BODY% Only put BODY into the document if externalize is on
	}{%
	\ifthenelse{\equal{#1}{}}{
		\includegraphics{\externalFolder/#2.pdf}%
	}{% width provided
	\includegraphics[width=#1]{\externalFolder#2.pdf}%
}%
}%
}{}%

\externalOnly{
	\usepackage{tikz,pgfkeys,pgfplots,tikz-3dplot}
	\usetikzlibrary{external}
	\tikzexternalize[prefix=\externalFolder]
	\usetikzlibrary{spy,calc,external,arrows}
	\usetikzlibrary{decorations.markings,positioning}
	\usepackage{letltxmacro}
	\tikzstyle{point}=[inner sep=3ptpt, outer sep=0pt,fill=black]%
	
\def\addlegendimage{\csname pgfplots@addlegendimage\endcsname}
\pgfplotsset{
	/pgfplots/uxbox/.style ={
		legend image code/.code={
			\draw[thin] (-0.08cm,-0.08cm) rectangle ++(0.16cm,0.16cm); 	
			\draw[fill] (0.4,-0.08cm) rectangle ++(0.16cm,0.16cm); 		
			\draw[->] (0,0)-- ++(0.3cm,0);
			\draw[->] (0.48,0)-- ++(0.3cm,0);
		}
	}
}
\pgfplotsset{
	/pgfplots/uybox/.style ={
		legend image code/.code={
			\draw[thin] (-0.08cm,-0.08cm) rectangle ++(0.16cm,0.16cm);
			\draw[->] (0,0)-- ++(0,0.3cm);			
			\draw[fill] (0.4,-0.08cm) rectangle ++(0.16cm,0.16cm); 			
			\draw[->] (0.48,0)-- ++(0,0.3cm);
		}
	}
}
\pgfkeys{/pgfplots/number in legend/.style={%
		/pgfplots/legend image code/.code={%
			\node at (0.295,-0.0225){#1};
		},%
	},
}
}

\DeclareMathOperator{\RR}{\mathbb{R}}

\hypersetup{pdfauthor={Sebastian Neumayer, Johannes Persch, Gabriele Steidl},
	pdftitle={Morphing of Manifold-Valued Images inspired by Discrete Geodesics in Image Spaces },
	pdfsubject={Draft},%
	pdfcreator = {pdflatex and TexStudio},%
	pdfkeywords={},
	plainpages=false, pdfstartview=FitH, pdfview=FitH, pdfpagemode=UseOutlines,%
	bookmarksnumbered=true,bookmarksopen=false,bookmarksopenlevel=0,%
	colorlinks=true,linkcolor=black,citecolor=black,urlcolor=black%
}
\begin{document}
\title{Regularization of Inverse Problems via\\
Time Discrete Geodesics in Image Spaces 
}

\author{
	Sebastian Neumayer\footnotemark[1] \and Johannes Persch\footnotemark[1] \and Gabriele Steidl\footnotemark[1] \footnotemark[2]}

\maketitle
\footnotetext[1]{Department of Mathematics,
	Technische Universität Kaiserslautern,
	Paul-Ehrlich-Str.~31, D-67663 Kaiserslautern, Germany,
	\{sneumaye,persch,steidl\}@mathematik.uni-kl.de.} 
\footnotetext[2]{Fraunhofer ITWM, Fraunhofer-Platz 1,
	D-67663 Kaiserslautern, Germany}

\begin{abstract}
	\noindent\small
	This paper addresses the solution of inverse problems in imaging given an additional reference image.
	We combine a modification of the discrete geodesic path model for image metamorphosis with a variational model,
	actually the $L^2$-$TV$ model, for image reconstruction.
	We prove that the space continuous model has a minimizer which depends in a stable way from the input data.
	Two minimization procedures which alternate over the involved sequences of deformations and images in different ways
	are proposed.
	The updates with respect to the image sequence exploit recent algorithms from convex analysis to minimize the $L^2$-$TV$ functional.
	For the numerical computation we apply a finite difference approach on staggered grids together with a multilevel strategy.
	We present proof-of-the-concept numerical results for sparse and limited angle computerized tomography 
	as well as for superresolution demonstrating the power of the method.
\end{abstract}

%----------------------------------------------
\section{Introduction} \label{sec:intro}
%----------------------------------------------
In certain applications it makes sense  to account for qualitative prior image information to improve the image reconstruction.
Typical examples are image superresolution and computerized tomography (CT)
with sparsely or limited angle sampled sinogram data.
Earlier approaches to incorporate prior knowledge on the image into CT include phase field methods \cite{BMVC2016,NSB2017},
the application of level set techniques, in particular when combining registration with segmentation \cite{SPSM2018},
as well as the utilization of local (shape) descriptors \cite{NAM2017,YZL2018}.
Recently, a mathematical classification of artifacts from arbitrary incomplete $X$-ray tomography data 
using the classical filtered backprojection was given in \cite{BFJQ2018}.
For earlier papers on the this topic the reader may also consult
\cite{FQ2013,FQ2016,Ka1997,Ngu2015}.

In this paper, we incorporate a whole reference image into the reconstruction process 
and take its deformation towards the image of interest, which is only indirectly given by measurements, 
into account.
Recent work in this direction shows promising results.
Schumacher, Modersitzki and Fischer \cite{SMF09}
have dealt with the combined reconstruction and motion
correction in SPECT imaging.
Karlsson and Ringh \cite{KR17} coupled the optimal transport model with inverse problems.
Chen and \"Oktem \cite{CO18} tackled hard inverse problems with shape priors 
under the name \textit{indirect image registration} 
within the large deformation diffeomorphic metric mapping (LDDMM) framework
and in an earlier paper \cite{OECDRB18} 
via linearized deformations. 
The authors use ODE constrained problem formulations, 
where the regularization of the deformations exploits reproducing kernel Hilbert spaces.
As a drawback, the LDDMM \cite{BMTY2005,CRM96, DGM98,Tro95,Tro98} based methods can only deal with images having the same intensities.
The metamorphosis model of Miller, Trouv\'e and Younes \cite{MY2001,TY2005a,TY2005b}
is an extension of the LDDMM approach which allows the variation of the image intensities along trajectories of the pixels.
A comprehensive overview over the topic is given in the book \cite{Younes2010} as well as in the review article \cite{MTY15}.
For a historic account see also \cite{MTY02}.
In a recent preprint, Gris, Chen and \"Oktem \cite{GO18} have enlarged the ideas in \cite{CO18,OECDRB18} to the metamorphosis setting.

In our paper, we also follow the metamorphosis idea, but in a completely different way than in \cite{GO18}.
We built up on the time discrete geodesic calculus proposed for shape spaces by Rumpf and Wirth \cite{RW13,RW15}
and for images by Berkels, Rumpf and Effland~\cite{BER15}. 
For convergence of the time discrete path model to the metamorphosis one we refer to these papers.
Here deformations are modeled via a smoothness term and the linearized elastic potential, which is
also a usual choice in registration problems.
We combine this model with  a ,,usual'' variational image reconstruction model,
actually the $L^2$-$TV$ model, which originated from \cite{ROF92}. 
Inspired by compressive sensing \cite{CRT2006,Don2006} such variational image reconstruction techniques 
with sparsity-exploiting priors have achieved impressive reductions in sampling requirements.
Besides TV priors, 
%see \cite{YWN2017} for an $\ell_0$ adaptation, 
wavelet, shearlet- and curvelet representations \cite{CEGL2010,Fri2013} were exploited 
in CT reconstructions with incomplete data.

Let ${\mathcal X},{\mathcal Y}$ be Hilbert spaces and $A \in L({\mathcal X}, {\mathcal Y})$ a linear, continuous operator.
A typical space ${\mathcal X}$ will be the space of square integrable function $L^2(\Omega)$ defined over some image domain $\Omega \subset \mathbb R^2$.
We want to reconstruct an unknown image $I_{\mathrm{orig}} \in {\mathcal X}$ 
having the following information available:
\begin{itemize}
\item[I1)] an image $B = A I_{\mathrm{orig}} + \eta \in {\mathcal Y}$, where $\eta$ denotes some small error, e.g.~due to noise.
\item[I2)] a reference image $R$ which is similar to the original image.
\end{itemize}
A usual variational model to approximate $I_{\mathrm{orig}}$ from $B$ using only I1) is given by 
\begin{equation} \label{mod:inverse}
\argmin_{I \in {\mathcal X}} {\mathcal E}(I;B)  \coloneqq {\mathcal D} (I;B) + \alpha {\mathcal P} (I), \quad \alpha \ge 0,
\end{equation}
where ${\mathcal D}$ is a data term and ${\mathcal P}$ a prior or regularizer.
A model for edge-preserving image reconstruction is the $L^2$-$\TV$ model, which will be  our model of choice.

To incorporate the reference image $R$, we want to combine model \eqref{mod:inverse} with a modified version of the time discrete geodesic
model for image metamorphosis \cite{BER15}. 
Given a template image $I_0 = T$ and a reference image $I_K = R$, this model aims
to find a chain of smooth deformations $(\varphi_0,\ldots,\varphi_{K-1})$
from an appropriately defined admissible set $\diffeo$ 
together with a sequence of images $(I_1,\ldots,I_{K-1})$
such that the sum of the quadratic distances
\begin{equation}\label{image_approx}
\sum_{k=0}^{K-1} \left \|I_{k} \circ \varphi_k^{-1} - I_{k+1} \right \|^2_{{\mathcal X} }
\end{equation}
together with a prior 
\[\sum_{k = 0}^{K-1}\int_\Omega \nu\lvert D^m\varphi_k\rvert^2 + W(D\varphi_k)\, \dx x, \quad \nu > 0,\]
on the deformations becomes small, see Fig.~\ref{fig:sequence}.
The first part of the deformation regularization enforces the smoothness of the mappings, while the
second term, circumscribed by $W$, will be chosen as linearized elastic potential.
By \eqref{image_approx}, the image sequence $(T,I_1,\ldots,R)$
may differ from  the deformed image sequence  
$(T , T \circ \varphi_0^{-1}, \ldots, T \circ \varphi_0^{-1} \circ \ldots \circ \varphi_{K-1}^{-1})$,
which makes the model flexible  for intensity changes.

\begin{figure}[t]
	\centering
	\includegraphics[width = .98\textwidth]{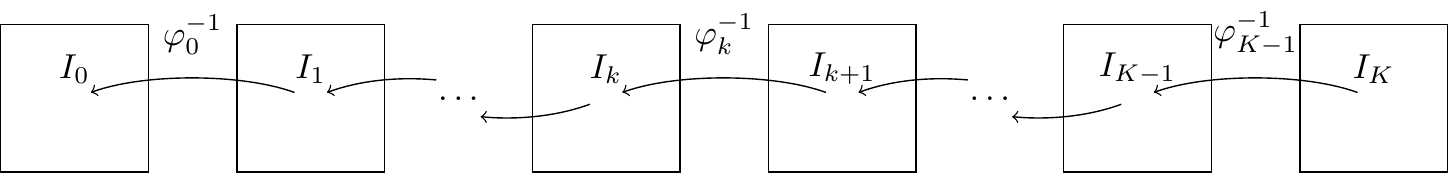}%	
	\caption{
	Illustration of the image and diffeomorphism path, where 
	$I_{k+1}(x) \approx I_k (\varphi_k^{-1}(x))$,
	$k=0,\ldots,K-1$. \label{fig:sequence}}
\end{figure}

For the numerical solution of our model 
we propose two different procedures, namely
proximal alternating linearized minimization (PALM) \cite{BST14} 
and an alternating minimization approach related to \cite{BER15,NPS17}. 
%Both methods seem to give similar results, but have differences in computation times. 
For the later one, recent primal-dual minimization algorithms from convex analysis 
are merged with a Quasi-Newton approach from image registration.

\paragraph{Outline of the Paper}
In Section \ref{sec:prelim}, the necessary preliminaries 
concerning the spaces of deformations and images are introduced.
In particular, we highlight properties of the concatenations of admissible deformations and $L^2$ images.
This motivates the modification of the time discrete path model \cite{BER15} 
and also of our generalized model for manifold-valued images in \cite{NPS17}.
In Section \ref{sec:model_cont}, our space continuous reconstruction model is established.
Since it combines \textbf{t}ime \textbf{d}iscrete \textbf{m}orphing with \textbf{inv}erse problems we call it TDM-INV.
We prove that the functional has a minimizer and that the minimizer depends stably on the input data.
Further, a convergence result for decreasing noise is provided.
Section \ref{sec:disc_minimization} deals with two minimization procedures.
For minimizing the image sequence we incorporate
primal-dual algorithms  from convex analysis.
Further, we explain computational issues in the space discrete setting.
The numerical examples in Section \ref{sec:numerics} demonstrate the very good performance of our algorithm.
We finish with conclusions in Section \ref{sec:conclusions}.

%-----------------------------------------------------
\section{Preliminaries} \label{sec:prelim}

In the rest of this paper, let $\Omega \subset \mathbb R^n$ 
be a nonempty, open, connected, and bounded set with Lipschitz boundary.
In this section, we introduce admissible sets $\diffeo$ of deformations 
and consider the concatenation of deformations $\varphi \in \mathcal A$ with images $I\in L^2(\Omega)$.
Note that $I \circ \varphi$ considered in \cite{BER15} is in general not in $L^2(\Omega)$ while we will see that
$I \circ \varphi^{-1} \in L^2(\Omega)$.
Therefore, we prefer to modify the time discrete geodesic path model by using the later concatenation.
Moreover, this fits better to the original metamorphosis setting of Tr\'ouve and Younes.
In \cite{Effland17} the image space $L^\infty(\Omega)$ is proposed instead and in \cite{NPS17} the computations are considerably
simplified by using a set $\mathcal{A}_\epsilon$ with deformations
fulfilling $\mathrm{det} (D\varphi) \ge \varepsilon$ for some fixed $\varepsilon >0$.

%-----------------------------------------------------
 \subsection{Admissible Deformations}\label{sec:admiss}	
%-----------------------------------------------------
First, we introduce the smoothness spaces of our deformation mappings.
Let $C^{k,\alpha}(\overline \Omega)$, $k \in \mathbb N_0$,  denote the H\"older space of functions $f \in C^k(\overline \Omega)$
for which 
	\[
	\|f\|_{C^{k,\alpha}(\overline \Omega)} 
	\coloneqq \sum_{|\beta| \le k} \|D^\beta f\|_{C(\overline \Omega)} 
	+  \sum_{|\beta| = k}   \sup_{\substack{x,y \in \Omega \\ x \not = y}} \frac{\bigl|D^\beta f(x) - D^\beta f(y)\bigr|}{|x-y|^\alpha}
	\]
	is finite. Equipped with this norm $C^{k,\alpha}(\overline \Omega)$ is a Banach space.
	
	By $W^{m,p}(\Omega)$, $m \in \mathbb N$, $1 \le p < \infty$, we denote the Sobolev space of functions having weak derivatives up to order $m$ 
	in $L^p(\Omega)$ with norm
	\begin{equation}
	\lVert f\rVert_{W^{m,p}(\Omega)}^p \coloneqq \sum_{\lvert\alpha\rvert \le m} \int_{\Omega} \lvert D^{\alpha} f\rvert^p \dx x
	\end{equation}
	and semi-norm
	$|D^m f|^p \coloneqq \sum_{\lvert\alpha\rvert = m} \lvert D^{\alpha} f\rvert^p$.
	For vector valued $F = (f_\nu)_{\nu=1}^n$, the component wise norm 
	$
	|D^m F|^p \coloneqq \sum_{\nu =1}^n \lvert D^m f_\nu \rvert^p
	$ is used.
	The space $W^{m,2}(\Omega)$ with
	$m> 1 + \frac{n}{2}$ is of particular interest, since it is compactly embedded in 
	$C^{1,\alpha}(\overline \Omega)$ for all $\alpha \in (0,m-1-\frac{n}{2})$ \cite[Theorem~8.13]{alt2002}
	and consequently also $W^{m,2}(\Omega)  \hookrightarrow W^{1,p}(\Omega)$ for all $p \ge 1$.
	
	It is assumed that the deformations $\varphi$ are elements of the
	following \textit{admissible set}
	\begin{equation}
	\diffeo \coloneqq \left\{\varphi\in \left(  W^{m,2}(\Omega) \right)^n\colon \det(D\varphi) > 0 \text{ a.e.~in } \Omega, \; \varphi(x) = x  \; \mathrm{for} \; x \in \partial\Omega \right\},
	\end{equation}
	where $m> 1 + \frac{n}{2}$.	
	Then, by a result of Ball \cite{Ball1981}, $\varphi$ has the following useful properties
  \begin{itemize}
  \item[i)] $\varphi(\overline{\Omega}) = \overline{ \Omega}$.
  \item[ii)] $\varphi$ maps measurable sets in $\overline{ \Omega}$ to measurable sets in $\overline{ \Omega}$ %holds in particular for Borel sets
  and the change of variables formula
  \[
  \int_{B} I \circ \varphi \det (D\varphi) \, dx = \int_{\varphi(B)} I \, dy
  \]
  holds for any measurable set $B \subset \overline \Omega$ and any measurable function $I\colon \overline \Omega \rightarrow \mathbb R$
  provided that one of the above integrals exists.
   \item[iii)] 
   $\varphi$ is injective a.e., i.e., the set
   \[
   S \coloneqq \left\{ x \in \overline \Omega\colon \varphi^{-1} (x) \; \mbox{has more than one element}\right\}
   \]
   has Lebesgue measure zero.
   \end{itemize}
By  property i) and since $\overline \Omega$ is  bounded, it follows immediately  
	for all $\varphi \in \diffeo$ that
	\begin{equation} \label{sieben}
	\| \varphi \|_{(L^\infty(\Omega))^n} \le C, \qquad \| \varphi \|_{(L^2(\Omega))^n} \le  C,
	\end{equation}
	with constants depending only on $\Omega$. 
	By the embedding properties of Sobolev spaces it holds $\varphi \in (C^{1,\alpha}(\overline \Omega))^n$.
	Further, by the inverse mapping theorem,  $\varphi^{-1}$ exists locally around a.e.~$x \in \Omega$ 
	and is continuously differentiable on the corresponding neighbourhood.
	However, to guarantee that $\varphi^{-1}$ is continuous (or, even more, continuously differentiable) on $\Omega$ further assumptions
	are required, see \cite[Theorem 2]{Ball1981}.
	A possible counterexample is the function $\varphi(x) \coloneqq x^3$ on $\Omega \coloneqq (-1,1)$, which is in $\diffeo$ but 
	$\varphi^{-1} = \mathrm{sgn} (x) |x|^\frac13$
	is not continuously differentiable.
	
	%------------------------------------------------------
	\subsection{Space of Images}\label{subsec:image}
	%------------------------------------------------------
	In this paper, we consider images as functions in ${\mathcal X} = L^2(\Omega)$.
	Unfortunately, the concatenation of $I \in L^2(\Omega)$ with $\varphi \in \diffeo$ can result in a function
	\[I \circ \varphi \not \in  L^2(\Omega),\]
	as the example 
	$I(x) \coloneqq x^{-\frac14}$ in $L^2((0,1))$ and $\varphi(x) \coloneqq x^2$ shows.
	However, this can be avoided by using
	\begin{equation} \label{alles_ok}
	\varphi \circ I \coloneqq  I \circ \varphi^{-1} \in L^2(\Omega),
	\end{equation} 
	where the function needs to be defined properly.
	To this end, let $\mathcal{N}$ be a Borel null set containing $S$ from iii).
	Then $\mathcal{B} \coloneqq \Omega \backslash \mathcal{N}$ is a Borel set with $\mu(\mathcal{B}) = \mu(\Omega)$.
	Note that $\varphi^{-1} (\mathcal{B})$ is itself a Borel set since $\varphi \in (W^{m,2}(\Omega))^n$ is measurable.
	Consider $\varphi^{-1}\colon \mathcal{B} \rightarrow \varphi^{-1} (\mathcal{B})$ and let $B \subseteq \varphi^{-1} (\mathcal{B})$ be a Borel set.
	Then, by ii), we see that $(\varphi^{-1})^{-1} (B) = \varphi(B)$ is a Borel set, so that $\varphi^{-1}$ is a  measurable function on $\mathcal{B}$.
	For $I \in L^2(\Omega)$ and $\varphi^{-1}$ as above,
	the concatenation $I \circ \varphi^{-1} \colon \mathcal{B} \to \mathbb R$ is measurable if defined as follows
	\begin{equation} \label{def_images}
	 I \circ \varphi^{-1} (x) \coloneqq 
	 \left\{
	 \begin{array}{ll}
	 I \circ \varphi^{-1} (x) & x \in \mathcal{B},\\
	 0                                                    &\text{otherwise.}
	 \end{array}
	 \right.
	\end{equation}
    Then, \eqref{alles_ok} can be verified by
	\[
	\int_{\Omega} \bigl|I \circ \varphi^{-1}\bigr|^2 \, \dx x = \int_{\mathcal{B}} \bigl|I \circ \varphi^{-1}\bigr|^2 \, \dx x
	= 
	\int_{\varphi^{-1}(\mathcal{B})} |I|^2 \det (D \varphi) \, \dx y,
	\]
	which is finite since $D \varphi$ has components in $C^{0,\alpha} (\overline \Omega)$.
	The same argument can be used to show that  $I \circ \varphi^{-1} \in L^p(\Omega)$, $p \in [1,\infty)$
	if
	$I \in L^p(\Omega)$.
	Further, the following lemma on the image of null sets under the deformations $\varphi$ and $\varphi^{-1}$ is useful.
\begin{lemma}\label{lem:null sets}
 For $\varphi \in \mathcal{A}$,  both $\varphi$ and its pre-image deformation $\varphi^{-1}$ 
 map null sets to null sets.
\end{lemma}

\begin{proof}
 	Since $\varphi$ is Lipschitz continuous, it maps null sets to null sets \cite[Theorem~3.33 and it's proof]{Wheeden2015}.
 	Now assume that there exists a Borel null set $\mathcal{N}$ %and then you can take any other null set. 
	with  $\mu( \varphi^{-1} (\mathcal{N}) )> 0$.
	Using the characteristic function $1_{\mathcal{N}}$ on $\mathcal{N}$, we get the contradiction
	\[
	0 = \int_{\mathcal{N}} 1_{\mathcal{N}} \, \dx x = \int_{\varphi^{-1}(\mathcal{N})} 1_{\mathcal{N}} \circ \varphi \,  \det(D \varphi) \, \dx y
	= \int_{\varphi^{-1}(\mathcal{N})} \det(D \varphi) \, \dx y > 0.
	\]
\end{proof}
Finally, we prove a continuity result for the $L^p(\Omega)$ norm with respect to mappings $\varphi \in \diffeo$.

\begin{lemma}\label{stet_norm}
	Let $I \in L^p(\Omega)$, $p \in [1,\infty)$ and $\{ \varphi^{(j)}\}_{j \in \NN}$ 
	be a sequence of deformations $\varphi^{(j)} \in \mathcal A$
	with	$\lim_{j \to \infty}\Vert \varphi^{(j)} - \hat \varphi \Vert_{(C^{1, \alpha}(\Omega))^n} = 0$
	for some $\hat \varphi \in \mathcal A$. 
	Then it  holds
	\[\lim_{j \to \infty} \bigl\| I \circ (\varphi^{(j)})^{-1} -  I \circ \hat \varphi^{-1} \bigr \|_{L^p(\Omega)} = 0.\]
\end{lemma}

	\begin{proof}
	Since $I \circ \hat \varphi^{-1} \in L^p(\Omega)$, there exits
	a  sequence $\{ I_k\}_{k\in\NN}$ of uniformly continuous functions
	with 
	$\| I \circ \hat \varphi^{-1}- I_k \|_{L^p(\Omega)} \le \frac{1}{k}$. 
	Using the fact that $\varphi^{-1}$ maps null sets on null sets, we conclude
	\begin{align}
	&\bigl \| I\circ (\varphi^{(j)})^{-1} - I \circ \hat \varphi^{-1} \bigr\|_{L^p(\Omega)}\\
	= &\bigl \| I \circ (\varphi^{(j)})^{-1}- I_k \circ \hat \varphi \circ (\varphi^{(j)})^{-1} + I_k \circ \hat \varphi \circ (\varphi^{(j)})^{-1} - I_k + I_k - I \circ \hat \varphi^{-1}\bigr\|_{L^p(\Omega)} 
	\\
	\leq& \;
	\left(\int_{\Omega} |I - I_k \circ \hat \varphi|^p  
	\,  \det \bigl( D\varphi^{(j)} \bigr) \, \dx x \right)^{\frac{1}{p}} + \left(\int_{\Omega} \bigl|I_k \circ \hat \varphi - I_k \circ \varphi^{(j)}\bigr|^p \,  
	\det \bigl( D\varphi^{(j)} \bigr)\, \dx x \right)^{\frac{1}{p}} 
	+  \frac{1}{k}.
	\end{align}
Due to the convergence of $\varphi^{(j)}$, 
there exists a constant $C$ such that
$\det (D\varphi^{(j)}) \le C$ for all $j \in \mathbb N$. 
Thus,
	\begin{align}
	   & \bigl\| I\circ (\varphi^{(j)})^{-1} - I \circ \hat \varphi^{-1} \bigr\|_{L^p(\Omega)}\\
	\leq& \;
	 \left(\int_{\Omega} |I - I_k \circ \hat \varphi|^p  \,  \det \bigl( D\varphi^{(j)} \bigr) \dx x \right)^{\frac{1}{p}} 
	+ C \left(\int_{\Omega} \bigl|I_k \circ \hat \varphi - I_k \circ \varphi^{(j)}\bigr|^p \dx x \right)^{\frac{1}{p}} 
	+  \frac{1}{k}.
	\end{align}
The last term converges to zero as $k \rightarrow \infty$. Now fix $k \in \mathbb N$.
Since $\varphi^{(j)}$ converges uniformly to $\hat \varphi$, the uniform continuity of $I_k$ can be used 
to conclude that $I_k\circ \varphi^{(j)}$ converges uniformly to $I_k \circ \hat \varphi$. 
Then boundedness of $\Omega$ implies that the second term converges to zero as $j \to\infty$. 
For the first term the uniform continuity of $I_k$ implies that for every $\epsilon >0$ there exits $j \in \mathbb N$
large enough such that
	\begin{align}
	\left(\int_{\Omega}| I - I_k \circ \hat \varphi|^p \det \bigl( D\varphi^{(j)} \bigr) \, \dx x \right)^{\frac{1}{p}}
	&\leq 
	\left(\int_{\Omega} |I - I_k \circ \hat \varphi|^p \det (D\hat \varphi) \dx x \right)^{\frac{1}{p}} + \epsilon\\
	& = \left(\int_{\Omega} \left|I \circ \hat \varphi^{-1} - I_k \right|^p \, \dx x \right)^{\frac{1}{p}} + \epsilon \le  \frac{1}{k} + \epsilon.
	\end{align}
	This concludes the proof.
	\end{proof}

%-----------------------------------------------------
\section{Space Continuous Model} \label{sec:model_cont}
%-----------------------------------------------------
In this section, we establish our space continuous model, which takes the information I1) and I2) into account
and prove existence of minimizers, stability and convergence for vanishing noise. These three properties are necessary for a well-defined regularization method.

%-----------------------------------------------------
\subsection{Model} \label{subsec:model}
%--------------------------------------------------------
Starting with the information I1), we are interested in reconstruc\-ting a two-dimensional image  from its measurements based on the variational approach \eqref{mod:inverse}. 
In this paper, the main focus lies on the total variation semi-norm as regularizer ${\mathcal P}$.
More precisely,  recall that the space of functions of bounded variation $BV(\Omega)$ consists of those functions $I \in L_{\mathrm{loc}}^1(\Omega)$ having weak first order derivatives which are finite Radon measures.
For $I \in L^1(\Omega)$, it holds that $I \in BV(\Omega)$ if and only if 
\[
TV(I) \coloneqq \sup \left\{ \int_\Omega I \mathrm{div} (\eta) \, \dx x\colon \eta \in \bigl(C^\infty_0(\Omega) \bigr)^n, \, |\eta| \le 1 \right\} < +\infty.
\]
The space $BV(\Omega)$ becomes a Banach space with the norm
$\|I\|_{BV} \coloneqq \|I\|_{L^1(\Omega)} + TV(I)$.
For $\Omega \subset \mathbb R^2$, i.e.~$n=2$, the space $BV(\Omega)$ can be  continuously embedded into $L^2(\Omega)$, see \cite[Theorem~3.47]{AFP2000}. 
Therefore, we can define
	\begin{equation} \label{BV}
	{\mathcal P} (I) \coloneqq 
	\left\{
	\begin{array}{ll}
	TV(I)& \mathrm{for} \; I \in BV(\Omega),\\
	+\infty & \mathrm{for} \; I \in L^2(\Omega)\backslash BV(\Omega).
	\end{array}
	\right.
	\end{equation}
	It is well-known that ${\mathcal P}$ in \eqref{BV} is a proper, convex and lower semi-continuous (lsc) functional on $L^2(\Omega)$,
	see \cite[Proposition 10.8]{SGGHL09}.		

Let $A\colon L^2(\Omega) \rightarrow {\mathcal Y}$ be a continuous linear operator into a Hilbert space ${\mathcal Y}$
which does not vanish on constant functions
and $B \in {\mathcal Y}$.
In case of the Radon transform, it holds ${\mathcal Y} =  L^2(\mathbb S^1,(-1,1))$. 
Then, we define the variational reconstruction model
	\begin{equation}\label{ROF}
	{\mathcal E}(I;B) \coloneqq \frac12 \|A I - B\|^2_{\mathcal Y} + \alpha TV (I), \quad \alpha > 0.
	\end{equation}
Note that ${\mathcal E}(I;B)$ is jointly weakly lsc in $I$ and $B$.
	
Having a reference image $R \in	L^2(\Omega)$ available, we want to add information I2) to the model.
To this end, let $W\colon \mathbb R^{2,2} \rightarrow \mathbb R_{\ge 0}$ be a lsc mapping and $\nu >0$, $m > 2$. 
Throughout the paper, it is assumed that $K \geq 1$ is an integer. 
For a sequence ${\mathbf I} \coloneqq (I_0,\ldots,I_{K-1})$ of images in $L^2(\Omega)$ and 
a sequence of admissible deformations $\boldsymbol{\varphi} \coloneqq (\varphi_0,\ldots,\varphi_{K-1})$ we consider
the time discrete geodesic path model
\begin{equation} \label{path}
\mathcal{F}({\mathbf I}, \boldsymbol{\varphi}) 
\coloneqq 
\sum_{k = 0}^{K-1}\int_\Omega W(D\varphi_k)+\nu \bigl\lvert D^m\varphi_k\bigr\rvert^2
	+ \bigl| I_{k} \circ \varphi_{k}^{-1}  - I_{k+1}\bigr|^2 \dx x,
\end{equation}
where $I_K \coloneqq R \in L^2(\Omega)$ is a 
given reference image.
Then, our whole model reads as
\begin{align} \label{path_I_phi}
	{\mathcal J} ({\mathbf I}, \boldsymbol{\varphi}) 
	&\coloneqq
	\mathcal{E}(I_0;B) + \beta \mathcal{F}({\mathbf I}, \boldsymbol{\varphi}) \quad \mbox{subject to} \quad I_K = R,
\end{align}
where $\beta > 0$. We call this model TDM-INV model referring to 'time discrete morphing - inverse' problems.

\begin{remark}\label{lin_elastic_pot}
The linearized elastic potential is our choice for $W$ in \eqref{path}.
More precisely, rewriting the deformation as 
$\varphi(x) = x+v(x)$ %and the inverse is approximated by $\varphi(x)^{-1} \approx x-v(x)$. 
and introducing the notation of the (Cauchy) strain tensor of the displacement vector field 
	$v = (v_1,v_2)^\tT \colon \Omega\mapsto\RR^2$ as
	\[
	Dv_{\operatorname{sym}} \coloneqq 
	\begin{pmatrix} 
	\partial_x v_1 & \frac12 (\partial_y v_1 + \partial_x v_2)\\
	\frac12 (\partial_y v_1 + \partial_x v_2) & \partial_y v_2
	\end{pmatrix},
	\]
	we apply
	\begin{align}
	\mathcal{S}(v) &\coloneqq \int_{\Omega} \mu \trace \left( Dv_{\operatorname{sym}}^\tT \, Dv_{\operatorname{sym}} \right)
	+
	\frac{\lambda}{2}\trace \left( Dv_{\operatorname{sym}} \right)^2 \, \dx x, \quad \nu > 0 \label{eq:pot}.
	\end{align}
	Note that the linearized elastic potential is a usual regularizer in the context of registration, see \cite{HM06,Mod2004,PPS17}.
\end{remark}
%----------------------------------------------------------------
\subsection{Existence, Stability and Convergence} \label{subsec:ex}
%----------------------------------------------------------------
In this section, we prove that there exists a minimizer of $\mathcal{J}$ in \eqref{path_I_phi}. 
Based on this, we show 
its  stability with respect to the input data $B$
and 
the convergence of an image sequence $\{I_0^{(j)} \}_{j \in \mathbb N}$ obtained from minimizing
the functionals with input data $B_j$ fulfilling  $\| A I_0 - B_j\|_{\mathcal Y}^2 \le \delta_j$ 
for a zero sequence $\{ \delta_j \}_{j \in \mathbb N}$
and corresponding parameters $\alpha_j,\beta_j$  decaying faster than $\delta_j$ to $I_0$.

The existence proof is the hardest part.
As usual for functionals in two variables it is based on three pillars: 
First it is shown that a minimizer exists if one of the variables is fixed.
In a second step the results are merged to get the overall existence.

Fixing the image sequence $\mathbf{I}$ leads to the solution of single registration problems.
The proof of Lemma \ref{lem:ex:phi} follows similar ideas as in \cite{BER15,NPS17}.
However, since the setting in those papers is different, 
we prefer to carefully follow the lines and make the necessary modifications
to make the paper self-contained.
Fixing $\varphi$, it is necessary to deal with the additional term ${\mathcal E}$
and the proof of Lemma \ref{lem:uni:seq} is different from those in \cite{BER15,NPS17},
in particular it relies on nested weighted $L_2$ spaces.
Except for the first step, the existence proof of Theorem \ref{main}
requires completely new estimates compared to \cite{BER15,NPS17}.

To begin with, we fix an image sequence 
${\mathbf I} \in (L^2(\Omega))^{K}$ 
and show that $J({\mathbf I}, \cdot)$ 
has a minimizer $\boldsymbol{\varphi} \in \diffeo^K$. 
Then, the consideration can be restricted to $\mathcal{F}({\mathbf I}, \cdot)$ and it suffices to prove that each of the summands
\begin{equation}
	{\cal R} (\varphi_k;I_{k},I_{k+1}) \coloneqq \int_\Omega W(D\varphi_k) + \nu \bigl\lvert D^m\varphi_k \bigr\rvert^2\dx x,
	+ \bigl| I_{k} \circ \varphi_{k}^{-1}- I_{k+1}\bigr|^2 \dx x,
\end{equation}
for $k=0,\ldots,K-1$, has a minimizer in $\diffeo$.
	
%-------------------------------------------------------------------------------------------------------------------------	
\begin{lemma}\label{lem:ex:phi}
	Let $W\colon\RR^{n, n}\to\RR_{\ge 0}$ be a lsc mapping with the property
	\begin{equation} \label{prop3}
	W(M)=\infty \quad \mathrm{if} \quad \det M \le 0.
	\end{equation}
	Further, let $T,R\in L^2(\Omega)$ be given.
	Then there exists  a minimizer $\hat \varphi\in\diffeo$ of
	\begin{equation} %\label{reg}
		 \mathcal{R}(\varphi;T,R)  \coloneqq
		\int_{\Omega}W(D \varphi)+\nu \bigl\lvert D^m \varphi \bigr\rvert^2 + \left|T \circ\varphi^{-1} - R\right|^2  \dx x
	\end{equation}
	over all $\varphi \in \diffeo$.
\end{lemma}
%-------------------------------------------------------------------------------------------------------------------------

\begin{proof} 
	1. Let $\{ \varphi^{(j)} \}_{j\in\mathbb{N}}$, $\varphi^{(j)} \in \diffeo$, be a minimizing sequence
	of $\mathcal{R}$.  
	Then  it holds that $\mathcal{R}( \varphi^{(j)};T,R ) \leq C$ for all $j \in \mathbb N$.
	This implies that $\{ \varphi^{(j)}\}_{j\in\mathbb{N}}$ has uniformly bounded $(W^{m,2}(\Omega))^n$ semi-norm,
	and by \eqref{sieben} the sequence is also uniformly bounded in \\
	$(L^2(\Omega))^n$.	
	Now we apply the Gagliardo-Nirenberg inequality, see Remark~\ref{th:gnr}, which states that for all $0 \le i < m$ it holds
	\begin{equation}
	\bigl\lVert D^i \varphi_\nu^{(j)}\bigr\rVert_{L^2(\Omega)}
	\le C_1\bigl\lVert D^m \varphi_\nu^{(j)}\bigr\rVert_{L^2(\Omega)}
	+C_2\bigl\lVert \varphi_\nu^{(j)}\bigr\rVert_{L^2(\Omega)}, \quad \nu =1,\ldots,n.
	\end{equation}
	All terms on the right-hand side are uniformly bounded.
	Hence, the $(W^{m,2}(\Omega))^n$ norm of $\{ \varphi^{(j)}\}_{j\in\mathbb{N}}$ is uniformly bounded.
	Since $W^{m,2}(\Omega)$ is reflexive, there exists a subsequence 
	which converges weakly to some function $\hat \varphi$ in $(W^{m,2}(\Omega))^n$. 
	By the compact embedding $W^{m,2} (\Omega)\hookrightarrow C^{1,\alpha}(\overline{\Omega})$, 
	$\alpha \in (0,m-1-\frac{n}{2})$,
	this subsequence, which is again denoted by $\{ \varphi^{(j)}\}_{j\in\mathbb{N}}$, 
	converges strongly to $\hat \varphi$ in $( C^{1,\alpha}(\overline{\Omega}))^n$
	and hence $D \varphi^{(j)}$ converges uniformly to $D \hat \varphi$.
	
	2. Next we show that $\hat \varphi$ is in the set ${\cal A}$.
	Since $W$ is lsc, we conclude
	\[
	\liminf_{j\to\infty} W\bigl(D \varphi^{(j)}\bigr)(x) \ge W(D \hat \varphi)(x)
	\]
	for all $x \in \Omega$ and since $W$ is nonnegative Fatou's lemma implies
	\[
	\int_{\Omega} W (D \hat \varphi)\dx x \le \liminf_{j \rightarrow \infty} \int_{\Omega} W\bigl(D\varphi^{(j)} \bigr)\dx x \le C.
	\]
	By incorporating \eqref{prop3} this implies $\det( D \hat \varphi) > 0$ a.e.
	Further, the boundary condition is fulfilled so that $\hat \varphi\in\diffeo$.  

	3. It remains to show that $\hat \varphi$ is a minimizer of $\mathcal{R}(\varphi;T,R)$.
	By Lemma \ref{stet_norm}, it holds
	$\| T \circ (\varphi^{(j)})^{-1} - T \circ \hat\varphi^{-1}\|_{L^2(\Omega)} \to 0$ as $j \to \infty$, 
	so that by the continuity of the norm
	\[ 
	\big \| T\circ \hat\varphi^{-1} - R \big\|_{L^2(\Omega)} 
	= \lim_{j \to \infty} 
	\big\| T\circ (\varphi^{(j)})^{-1}- R\big\|_{L^2(\Omega)}.
	\]
	This together with the previous steps of the proof implies 
	that  the three summands in $\mathcal{R}$ are (weakly) lsc.
	Hence, we obtain
	\begin{align}
	\mathcal{R}(\hat \varphi;T,R) 
	&\le 
	\liminf_{j\to\infty}  \int_{\Omega}W\bigl(D \varphi^{(j)}\bigr) + \nu \bigl\lvert D^m \varphi^{(j)} \bigr \rvert^2
	+
	\left| T\circ (\varphi^{(j)})^{-1}  - R \right|^2 \dx x\\
	&= 
	\inf_{\varphi\in\diffeo}\mathcal{R}(\varphi;T,R),
	\end{align}
	which proves the claim.
\end{proof}

Next, we fix a sequence of  mappings $\boldsymbol{\varphi}\in \diffeo^K$ 
and ask for a minimizer of ${\mathcal J }(\cdot, \boldsymbol{\varphi})$.

\begin{lemma}\label{lem:uni:seq}
Let $A\colon L^2(\Omega) \rightarrow {\mathcal Y}$ be a continuous linear operator into a Hilbert space ${\mathcal Y}$
which does not vanish on constant functions, $B \in {\mathcal Y}$ and $R \in L^2(\Omega)$.
For fixed $\boldsymbol{\varphi} \in \diffeo^K$, 
there exists a unique image sequence 
${\mathbf I} \in (L^2(\Omega))^K$
which minimizes ${\mathcal J }(\cdot, \boldsymbol{\varphi})$. 
%Note that $\mathcal{E}(I_K;B) <\infty$ and hence $I_K\in\operatorname{BV}$.
\end{lemma}	
		
\begin{proof}
	We prove lower semi-continuity, coercivity and strict convexity of the functional.
	Neglecting the constant terms and by changing the indexing of the sum it remains to consider
	\begin{align} \label{path_I}
	J (\vec I) 
	\coloneqq 
	\beta \sum_{k = 1}^{K} \int_{\Omega} \bigl| I_{k-1} \circ \varphi_{k-1}^{-1} - I_{k} \bigr|^2 \dx x
	+
	\mathcal{E}(I_0;B) 
	\quad \mbox{subject to} \quad I_K = R.
	\end{align}
	Setting 
\begin{equation}\label{eq:psi}
\begin{split}
 \psi_0(y) &\coloneqq y,\\
 \psi_{k}(y) &\coloneqq \varphi_{k-1} \circ \psi_{k-1} (y) = \varphi_{k-1} \circ \ldots \circ \varphi_{0}(y), \quad k=1,\ldots,K,
 \end{split}
\end{equation}
and substituting $x \coloneqq \psi_k (y)$
in the $k$-th summand of \eqref{path_I}, the functional transforms to
\begin{align*} 
J ({\mathbf  I})
&= \beta \sum_{k=1}^K \int_{\Omega} | I_k \circ \psi_{k} - I_{k-1}\circ\psi_{k-1}|^2 \, \det \left( D \psi_k \right) \, \dx y 
+ {\mathcal E}(I_0;B).
\end{align*}
Using ${\mathbf F} \coloneqq (F_0,\ldots,F_{K-1})$, where
$F_0 \coloneqq I_0$,
$F_k \coloneqq I_k \circ \psi_{k}$, 
and 
$w_0(x) \coloneqq 1$, $w_k(x) \coloneqq  \det\left( D \psi_k(x) \right)$, we are concerned with the minimization of
\begin{align} \label{eq:aha}
\begin{split}
\tilde J ({\mathbf  F}) \coloneqq & \; \beta \sum_{k=1}^K \int_{\Omega}  \left| F_k - F_{k-1}\right|^2 \, w_k \, \dx  x 
+ {\mathcal E}(F_0;B)\\
&\mbox{subject to} \quad F_K = R\circ \psi_K.
\end{split}
\end{align}
Note that by $0< w_k \le C$ a.e.~and $w_{k} = w_{k-1} \det (D \varphi_{k-1}\circ \psi_{k-1})$,
the weighted $L^2$ spaces are nested
\begin{equation} \label{schachtel}
L^2(\Omega) = L^2_{w_0}(\Omega) \subseteq L^2_{w_1}(\Omega) \subseteq \ldots \subseteq L^2_{w_K}(\Omega),
\end{equation}
in particular $F_K \in L^2_{w_K}(\Omega)$ if $R \in L^2(\Omega)$.
A minimizer must fulfill $F_0 \in BV(\Omega) \subset L^2(\Omega)$, 
and by
successively considering the integrals in \eqref{eq:aha} further
$F_k \in L^2_{w_k}(\Omega)$.
In the following, we set $\beta \coloneqq 1$ to simplify the notation.
Since the function $g\colon \mathbb R^{K} \rightarrow \mathbb R$,
\[
g(f_0,\ldots,f_{K-1}) \coloneqq \sum_{k=1}^K (f_k - f_{k-1})^2 w_k
\]
with $w_k >0$ and $f_K$ fixed, is strictly convex, the sum of the integrals in $\tilde J$ is strictly convex.
Clearly, this sum can be rewritten as
$\sum_{k=1}^{K-1} \|F_k - F_{k-1}\|_{L^2_{w_k}(\Omega)} + \|R\circ \psi_K - F_{K-1}\|_{L^2_{w_K}(\Omega)}$
and is continuous.
Since ${\mathcal E}(F_0;B)$ is proper, convex and lsc, the same holds true for $\tilde J$ 
over $L^2_{w_0}(\Omega) \times \ldots \times L^2_{w_{K-1}}(\Omega)$. Thus, $\tilde J$ is also weakly lsc \cite[Lemma~10.4]{SGGHL09}.

Next we show that $\tilde J$ is coercive.
Assume conversely that $\sum_{k=0}^{K-1} \|F_k^{(j)}\|_{L^2_{w_k}(\Omega)} \rightarrow \infty$
but $\tilde J(\mathbf F^{(j)})$ is bounded. 
By the assumptions on $A$ it holds that ${\mathcal E}(F_0;B)$ is coercive, see \cite[Theorem 6.115]{BL2011}.
Thus, $\|F_0^{(j)}\|_{L^2(\Omega)}$ is bounded and by \eqref{schachtel} also $\|F_0^{(j)}\|_{L^2_{w_1}(\Omega)}$ is bounded.
Considering successively the integrals in \eqref{eq:aha} we obtain that $\|F_k^{(j)}\|_{L^2_{w_k}(\Omega)}$, $k=1,\ldots,K-1$
is bounded which contradicts our assumption.

Thus, $\tilde J$ is coercive and since it is weakly lsc and strictly convex, the functional has a unique minimizer $\mathbf F$.
By definition of $\mathbf F$ the unique minimizer of $J$ is given by $\mathbf I$ with $I_k = F_k \circ \psi_k^{-1} \in L^2(\Omega)$.
\end{proof}

%-------------------------------------------------------------------------
For our computations, the following corollary on the minimizer of
$\tilde J$ in \eqref{eq:aha} with fixed $F_0$ will be useful.

\begin{corollary}\label{cor:tridiag}
Let $K \ge 2$ be an integer.
Further, let $w_k \in C^{0,\alpha}(\overline \Omega)$, $k=1,\ldots, K$ fulfill $w_k > 0$ a.e.~on $\Omega$ and 
$\frac{w_{k+1}}{w_{k}} \le C$, $k=1,\ldots,K-1$.
For given $F_0 \in L^2(\Omega)$ and $F_K \in L^2_{w_K}(\Omega)$,
the  solution of
\[
\argmin_{F_k \in L^2_{w_k}(\Omega)} \sum_{k=1}^{K} \int_{\Omega}  \left| F_k - F_{k-1} \right|^2 \, w_k \, \dx  x 
\]
is given by
\begin{equation} \label{aha}
F_k = t_k F_0 + (1-t_k) F_K, \quad t_k \coloneqq \frac{\sum_{i=1}^k w_i^{-1}}{\sum_{i=1}^K w_i^{-1} }.
\end{equation}
\end{corollary}

\begin{proof}
Setting the first derivative of the functional to zero we obtain a.e.~on $\Omega$,
\[
w_k (F_k - F_{k-1}) + w_{k+1} (F_k - F_{k+1}) = 0, \quad k=1,\ldots,K-1.
\]
This can be rewritten as linear system of equations
\[
\mathrm{tridiag} (-w_k,w_k + w_{k+1},-w_{k+1})_{k=1}^{K-1} (F_1,\ldots,F_{K-1}) ^\tT =
(w_1 F_0,0,\ldots,0, w_K F_{K} )^\tT.
\]
Since the tridiagonal matrix is irreducible diagonal dominant, the system has a unique solution.
Straightforward computation shows that the solution is given by \eqref{aha}.
\end{proof}

Now we can prove the three main results of this section, beginning with existence of minimizers.
	
\begin{theorem}[Existence]\label{main}
		Let $R\in L^2(\Omega)$ and $B \in \mathcal{Y}$ . 
		Then there exists $(\hat{\mathbf I},\hat{\boldsymbol{\varphi}} ) \in L^2(\Omega)^{K} \times \diffeo^K$ 
		minimizing $\mathcal{J}$.	
\end{theorem}

\begin{proof}
		The outline of the proof is as follows. 
		First, we take a minimizing sequence of $\mathcal J$ 
		and show that the deformations and the intermediate images have a weakly convergent subsequence. 
		Then, we prove that their concatenation is also weakly convergent 
		and use this to get the weak lower semi-continuity of the functional.
		
		1. Let 
		$\{({\mathbf I}^{(j)}, \boldsymbol{\phi}^{(j)})\}_{j\in\mathbb{N}}$
		be a minimizing sequence of  ${\mathcal J}$. 
		Then ${\mathcal J} ({\mathbf I}^{(j)},\boldsymbol{\phi}^{(j)} ) \le C$ for all $j\in\mathbb{N}$.
		By Lemma \ref{lem:ex:phi}, we find for each ${\mathbf I}^{(j)}$ 
		a sequence of diffeomorphisms $\boldsymbol{\varphi}^{(j)}$ 
		such that
		\[
		\mathcal{J}\bigl(\vec{I}^{(j)} ,\boldsymbol{\varphi}^{(j)}\bigr) 
		\le 
		\mathcal{J} \bigl(\vec{I}^{(j)},\boldsymbol{\varphi}\bigr) 
		\]
		for all $\boldsymbol{\varphi}\in \mathcal{A}^K$.
		Then, we know 
		$\lVert D^m \varphi_k^{(j)}\rVert_{L^2(\Omega)}^2<\frac{1}{\nu}C$ 
		for all $j\in\mathbb{N}$ and $k = 0,\dots,K-1$. 
		As in the first part of the proof of Lemma \ref{lem:ex:phi}
		we conclude that $\{ \varphi_k^{(j)} \}_{j\in\mathbb{N}}$ 
		is bounded in 
		$(W^{m,2}(\Omega))^n$,
		so that there exists a subsequence 
		converging weakly in $(W^{m,2}(\Omega))^n$ 
		and strongly in $(C^{1,\alpha}(\overline{\Omega}))^n$ to $\hat{\varphi}_k$. 
		Set $\hat{\boldsymbol{\varphi}} \coloneqq \left( \hat{\varphi}_k \right)_{k=0}^{K-1}$ and let us denote this subsequence again by $\{ \varphi_k^{(j)} \}_{j\in\mathbb{N}}$ 
		and define $\boldsymbol{\varphi}^{(j)} \coloneqq (\varphi^{(j)}_k)_{k=0}^{K-1}$.
		
		2.
		Since ${\mathcal J}(\vec{I}^{(j)} ,\boldsymbol{\varphi}^{(j)}) \le C$ for all $j \in \mathbb N$, coercivity of $\mathcal E$ implies that $\|{I}_{0}^{(j)} \|_{L^2(\Omega)}$ is bounded.
		Additionally, we conclude for $k=0,\ldots,K-2$ that
	\begin{align*}
	\bigl \|{I}_{k+1}^{(j)} \bigr \|_{L^2(\Omega)} 
	&\le 
	\bigl \| {I}_{k}^{(j)} \circ (\varphi_{k}^{(j)})^{-1} -  {I}_{k+1}^{(j)} \bigr \|_{L^2(\Omega)} 
	+ \bigl \|{I}_{k}^{(j)}\circ (\varphi_{k}^{(j)})^{-1} \bigr\|_{L^2(\Omega)}
	\\
	&\le C^\frac12 + \bigl \|{I}_{k}^{(j)}\circ (\varphi_{k}^{(j)})^{-1} \bigr \|_{L^2(\Omega)}.
	\end{align*}
	Further, $\varphi_k^{(j)}$ is convergent in $( C^{1,\alpha}(\overline{\Omega}))^n$
	and consequently $\det(D\varphi_k^{(j)}) \le \tilde C$ on $\Omega$ for $k=0,\ldots,K-1$.
	Then, it holds
	\begin{align*}
	\bigl \|{I}_1^{(j)} \bigr\|_{L^2(\Omega)} &\le C^\frac12 + \bigl \|I_0^{(j)} \circ (\varphi_0^{(j)})^{-1} \bigr \|_{L_2(\Omega)}\\
	\bigl \|{I}_2^{(j)} \bigr \|_{L^2(\Omega)} 
	&\le C^\frac12 + \bigl \| I_1^{(j)}\circ (\varphi_1^{(j)})^{-1} \bigr \|_{L^2(\Omega)}\\
	&= C^\frac12 + \left( \int_\Omega \bigl| {I}_1^{(j)} \bigr|^2 \,\det \bigl( D \varphi_1^{(j)} \bigr)  \dx x \right)^\frac12\\
	&\le C^\frac12 + {\tilde C}^{\frac12}\bigl\| {I}_1^{(j)}\bigr\|_{L^2(\Omega)}.  
	\end{align*}
	Successive continuation shows that the sequence $\{{\vec{I}}^{(j)} \}_{j\in\mathbb{N}}$ is bounded in $(L^2(\Omega))^{K}$. 
	Hence, there exists a weakly convergent subsequence, also denoted by $\{{\vec{I}}^{(j)} \}_{j\in\mathbb{N}}$, which converges to $\hat {\vec I} \in (L^2(\Omega))^{K}$.
	
	3. Next, we show the weak convergence of $I_k^{(j)} \circ (\varphi_k^{(j)})^{-1}$ to $\hat I_k \circ \hat \varphi_k^{-1}$. 
	Since the sequence is bounded, it suffices to test with $g \in C_c^\infty(\Omega)$. It holds
	\[
	\int_{\Omega} \Big(I_k^{(j)} \circ \big(\varphi_k^{(j)}\big)^{-1} - \hat I_k \circ \hat \varphi_k^{-1} \Big)g \, \dx x
	=
	{\mathcal I}_1^{(j)} + {\mathcal I}_2^{(j)}
	\]
	with
	\begin{align*}
	{\mathcal I}_1^{(j)} &\coloneqq\int_{\Omega} \Big(I_k^{(j)} \circ \big(\varphi_k^{(j)}\big)^{-1} - I_k^{(j)} \circ \hat \varphi_k^{-1}\Big)g \, \dx x,\\
	{\mathcal I}_2^{(j)} &\coloneqq \int_\Omega \Big(I_k^{(j)} \circ \big(\hat \varphi_k\big)^{-1} - \hat I_k \circ \hat \varphi_k^{-1} \Big)g \, \dx x.
	\end{align*}
	Using the change of variables formula, we obtain 
	\[
	{\mathcal I}_2^{(j)} = \int_\Omega \left( I_k^{(j)}  - \hat I_k  \right) \det \left( D \hat \varphi_k \right)  g \circ \hat \varphi_k \, \dx x.
	\]
	Since $ \det \left( D \hat \varphi_k \right)  g\circ \hat \varphi_k \in L^2(\Omega)$, the weak convergence of $I_k^{(j)}$ to $\hat I_k$ implies that ${\mathcal I}_2^{(j)}$ converges to zero as $j \rightarrow \infty$.
	Using the change of variables formula again, ${\mathcal I}_1^{(j)}$ can be estimated by
	\begin{align*}
	{\mathcal I}_1^{(j)}
	&=  \int_{\Omega} I_k^{(j)} \Big(g \circ \varphi_k^{(j)} \det \bigl( D \varphi_k^{(j)} \bigr) 
	- g \circ \hat \varphi_k \det \left( D \hat \varphi_k \right) \Big) \dx x\\
	&\leq  
	\bigl \Vert I_k^{(j)} \bigr \Vert_{L^2(\Omega)} 
	\bigl \Vert g \circ \varphi_k^{(j)}  \det \bigl( D \varphi_k^{(j)}  \bigr)
	- g \circ \hat \varphi_k \det \left( D \hat \varphi_k  \right) \bigr \Vert_{L^2(\Omega)}.
	\end{align*}
	Since $\{ I_k^{(j)} \}_{j\in\mathbb{N}}$ is bounded, it suffices to show the convergence of the second factor. 
	With $g_k^{(j)} \coloneqq g \circ \varphi_k^{(j)}$ and $\hat g_k \coloneqq  g \circ \hat \varphi_k$ it follows that
	\begin{align*}
	& \bigl \Vert g_k^{(j)}  \det \bigl( D \varphi_k^{(j)}  \bigr) -\hat g_k \det (D \hat \varphi_k) \bigr \Vert_{L^2(\Omega)}\\
	\leq& \; \bigl \Vert g_k^{(j)} \det \bigl( D \varphi_k^{(j)}  \bigr)-g_k^{(j)} \det \left( D \hat \varphi_k \right) \bigr \Vert_{L^2(\Omega)} 
	+ \bigl \Vert g_k^{(j)} \det \left( D \hat \varphi_k  \right) - \hat g_k\det \left( D \hat \varphi_k  \right) \bigr \Vert_{L^2(\Omega)}\\
	\leq& \; 
	C \bigl \Vert \det \bigl( D \varphi_k^{(j)} \bigr) -\det \left( D \hat \varphi_k  \right) \bigr \Vert_{C^0(\overline \Omega)} 
	+ C  \bigl \Vert g_k^{(j)} - \hat g_k  \bigr \Vert_{L^2(\Omega)}.
	\end{align*}
	The first term converges to zero since $D\varphi_k^{(j)}$ is convergent. 
	Uniform convergence of $\varphi_k^{(j)}$ together with the uniform continuity of $g$ implies that $g_k^{(j)}$ 
	converges uniformly to $\hat g_k$.
	Now boundedness of $\Omega$ implies that the second term converges to zero.
	
	4. It remains to show that $(\hat{\vec{I}} ,\hat{ \boldsymbol{\varphi}})$ is a minimizer of 
	${\mathcal J}(\vec{I},\boldsymbol{\phi})= {\mathcal E}(I_0;B) + \beta {\mathcal F}({\vec I}, \boldsymbol{\varphi})$. 
	It holds
	\begin{align*}
	\liminf_{j \to \infty} {\mathcal F} \bigl(\vec{I}^{(j)} ,\boldsymbol{\varphi}^{(j)}\bigr) 
	\ge&
	\sum_{k=0}^{K-1} \liminf_{j \to \infty} \int_{\Omega}  W \bigl( D\varphi_k^{(j)} \bigr) \dx x
	+ \nu \liminf_{j \to \infty}  \int_{\Omega} \bigl|D^m \varphi_k^{(j)} \bigr|^2 \dx x \\
	+ & \, 
	\liminf_{j \to \infty} \bigl\| I_k^{(j)} \circ \big(\varphi_k^{(j)}\big)^{-1} -  I_{k+1}^{(j)} \bigr\|_{L^2(\Omega)}^2.
	\end{align*}
	The components of $\varphi_k^{(j)}$ weakly converge in $W^{m,2}(\Omega)$, those of $D \varphi_k^{(j)}$
	converge in $C^{0}(\overline \Omega)$, and $I_k^{(j)}$, $I_k^{(j)}\circ(\varphi_k^{(j)})^{-1}$ weakly converges in $L^2(\Omega)$. 
	We use this together with the facts that the first summand is lsc, the second one weakly lsc and $\|f-g\|_{L^2(\Omega)}^2$ 
	is weakly lsc (convex and lsc) in both arguments to conclude
	\[\liminf_{j \to \infty} {\mathcal F} \bigl(\vec{I}^{(j)} ,\boldsymbol{\varphi}^{(j)}\bigr) 
	\ge {\mathcal F} \bigl(\vec{\hat I} ,\boldsymbol{\hat \varphi}\bigr) .\]
	Since ${\mathcal E}(I_0;B)$ is weakly lsc in $I_0$  we obtain
	\[
	\inf_{\vec{I},\boldsymbol{\phi}}  {\mathcal J} \bigl(\vec{I},\boldsymbol{\phi}\bigr) = 
	\liminf_{j \to \infty} {\mathcal J} \bigl(\vec{I}^{(j)} ,\boldsymbol{\phi}^{(j)}\bigr) 
	\ge
	\liminf_{j \to \infty} {\mathcal J} \bigl(\vec{I}^{(j)} ,\boldsymbol{\varphi}^{(j)}\bigr) 
	\geq 
	{\mathcal J} \bigl(\hat{\vec{I}} ,\hat{ \boldsymbol{\varphi}} \bigr). 
	\]
\end{proof}

Next, we prove that the minimizers of ${\mathcal J}$ depend stably on the input data $B$. 
To emphasize the dependence of $\mathcal J$ on $B$, we use the notation $\mathcal J_B$ instead of $\mathcal J$.

\begin{theorem}[Stability]\label{thm:stab}
	Let $R \in L^2(\Omega)$.
	Further, let $\{B_j\}_{j\in\mathbb{N}}$ be a sequence in $\mathcal Y$ converging to $B \in \mathcal Y$.
	For each $j \in \NN$, we choose a minimizer $({\mathbf I}^{(j)}, \boldsymbol{\varphi}^{(j)})$ of $\mathcal J_{B_j}$.
	Then, there exists a subsequence of $\{({\mathbf I}^{(j)}, \boldsymbol{\varphi}^{(j)})\}_{j\in\mathbb{N}}$ which converges weakly to a minimizer $(\hat{\vec I}, \hat{\boldsymbol{\varphi}})$ of $\mathcal J_{B}$.
\end{theorem}

\begin{proof}
	1. Due to the convergence of $\{B_j\}_{j \in \NN}$ it holds for every $I_0 \in L^2(\Omega)$ that
	\[\mathcal E(I_0;B_j) = \frac12 \|A I_0 - B_j\|^2_{\mathcal Y} + \alpha TV (I_0) \to \mathcal E(I_0,B).\]
	Hence, there exists $C>0$ with $\mathcal J_{B_j} ({\mathbf I}^{(j)},\boldsymbol{\varphi}^{(j)} ) \leq \mathcal J_{B_j} ({\mathbf I}^{(1)},\boldsymbol{\varphi}^{(1)} ) \le C$ for all $j\in\mathbb{N}$.
	By definition of $\mathcal J_{B_j}$ we obtain $\lVert D^m \varphi_k^{(j)}\rVert_{L^2(\Omega)}^2<\frac{1}{\nu}C$ for all $j\in\mathbb{N}$ and $k = 0,\dots,K-1$.
	As in the first part of the proof of Lemma \ref{lem:ex:phi} we conclude that there exists a subsequence converging weakly in $(W^{m,2}(\Omega))^n$ 
	and strongly in $( C^{1,\alpha}(\overline{\Omega}))^n$ to $\hat \varphi_k$.
	Set $\boldsymbol{\hat \varphi} \coloneqq (\hat \varphi_k)_{k=0}^{K-1}$.
	Let us denote this subsequence again by $\{\varphi_k^{(j)} \}_{j\in\mathbb{N}}$ 
	and define $\boldsymbol{\varphi}^{(j)} \coloneqq (\varphi^{(j)}_k)_{k=0}^{K-1}$.
	
	2. Next, we estimate
	\[C \ge \mathcal E(I_0^{(j)};B_j) \geq \frac12 (\|A I_0^{(j)} - B\|_{\mathcal Y} - \|B - B_j\|_{\mathcal Y})^2 + \alpha TV (I_0^{(j)}) ,\]
	so that the coercivity of $\frac12 \|A I_0 - B\|^2_{\mathcal Y} + \alpha TV (I_0)$ in $I_0$ implies the boundedness of $\{I_0^{(j)}\}_{j \in \mathbb N}$.
	Now, we can reproduce the Steps 2 and 3 from Theorem \ref{main}
	to see that there exists a weakly convergent subsequence, also denoted by $\{{\vec{I}}^{(j)} \}_{j\in\mathbb{N}}$, 
	which converges to $\hat{\vec I} \in \left(L^2(\Omega)\right)^{K}$.
	Additionally, the sequence $I_k^{(j)} \circ (\varphi_k^{(j)})^{-1}$ converges weakly to $\hat I_k \circ \hat \varphi_k^{-1}$.
	
	3. It remains to show that $(\hat{\vec I},\hat{\boldsymbol \varphi})$ minimizes ${\mathcal J}_B$. 
	We can use the lower semi continuity argument for $\mathcal F$ from Theorem~\ref{main} together with the fact that ${\mathcal E}(I;B)$ is jointly lsc to obtain for any $(\vec I, \boldsymbol{\varphi})$ that
	\[
	{\mathcal J_B} \bigl({\hat{\vec{I}}} ,\hat{\boldsymbol{\varphi}} \bigr) \leq
	\liminf_{j \to \infty} {\mathcal J_{B_j}} \bigl(\vec{I}^{(j)} ,\boldsymbol{\varphi}^{(j)}\bigr) \leq
	\liminf_{j \to \infty} {\mathcal J_{B_j}} \bigl(\vec{I} ,\boldsymbol{\varphi}\bigr) =
	{\mathcal J_{B}} \bigl(\vec{I} ,\boldsymbol{\varphi}\bigr). 
	\]
	The last equality follows from the convergence of $B_j$ together 
	with continuity of $\mathcal E$ in $B$. Hence, $(\hat{\vec I}, \hat{\boldsymbol{\varphi}})$ is a minimizer of $\mathcal J_B$.
\end{proof}

This section concludes with a convergence result for vanishing noise. 
Here we additionally need the dependence of $\mathcal J$ on the parameter $\alpha$ (for simplicity we choose $\alpha = \beta$) and hence we use $\mathcal J_{\alpha,B}$.

\begin{theorem}[Convergence]\label{thm:conv}
	Let $R\in L^2(\Omega)$ and $B \in \mathcal Y$, and suppose that there exists $(\tilde{\vec I}, \tilde{\boldsymbol{\varphi}})$ such that $A\tilde{I_0} = B$ 
	and $\mathcal J_{1,B}(\tilde{\vec I}, \tilde{\boldsymbol{\varphi}}) < \infty$.
	Further, assume that 
	$\alpha\colon \RR_{>0} \to \RR_{>0}$ satisfies $\alpha(\delta) \to 0$ 
	and 
	$\frac{\delta}{\alpha(\delta)} \to 0$ as $\delta \to 0$. 
	Assume that $\{\delta_j\}_{j\in\mathbb{N}}$ is a sequence of positive numbers converging to 0 
	and $\{B_j\}_{j\in\mathbb{N}}$ is a sequence in $\mathcal Y$ satisfying $\Vert B - B_j \Vert_\mathcal{Y}^2 \leq \delta_j$ 
	for each $j$.
	Let $({\mathbf I}^{(j)}, \boldsymbol{\varphi}^{(j)})$ be a minimizer of $\mathcal J_{\alpha_j, B_j}$, where $\alpha_j \coloneqq \alpha(\delta_j)$.
	Then, there exists a subsequence of $\{ I_0^{(j)}\}_{j\in\mathbb{N}}$ which weakly converges to an image $\hat I_0$ such that
	$A \hat I_0 = B$.
\end{theorem}

\begin{proof}
	For every $j \in \NN$, it holds
	\begin{align*}
	\bigl \Vert I_0^{(j)} \bigr \Vert^2_{L^2(\Omega)} 
	&\leq C\,TV\bigl(I_0^{(j)}\bigr) 
	\leq \frac{C}{\alpha_j} \mathcal J_{\alpha_j, B_j}({\mathbf I}^{(j)}, \boldsymbol{\varphi}^{(j)}) 
	\leq  \frac{C}{\alpha_j} \mathcal J_{\alpha_j, B_j} (\tilde{\vec I}, \tilde{\boldsymbol{\varphi}}) \\
	&=
	\frac{C}{2 \alpha_j} \| B - B_j\|_{\mathcal Y}^2 + C TV (\tilde {\vec I} ) + C \mathcal{F} (\tilde {\vec I}, \tilde {\boldsymbol{\varphi}})\\
	&\leq 
	\frac{C}{2} \frac{\delta_j}{\alpha_j} + C \mathcal{J}_{1,B} (\tilde{\vec I}, \tilde{ \boldsymbol{\varphi} }).
	\end{align*}
	From the assumptions on $\alpha$ and $\delta$ we deduce  that $\Vert I_0^{(j)} \Vert^2_{L^2(\Omega)}$ is bounded.
	Hence, there exists a weakly convergent subsequence with limit $\hat I_0$.
	Additionally, it holds $\Vert A\hat I_0 -B \Vert_{\mathcal Y}^2 \leq \liminf_{j \to \infty} \Vert AI_0^{(j)} - B_j \Vert_{\mathcal Y}^2$.
	Now we can estimate
	\begin{equation}
	\bigl \Vert AI_0^{(j)} - B_j \bigr \Vert_{\mathcal Y}^2 \leq {\mathcal J}_{\alpha_j, B_j} ({\mathbf I}^{(j)}, \boldsymbol{\varphi}^{(j)}) 
	\leq {\mathcal J}_{\alpha_j, B_j} (\tilde{\vec I}, \tilde{\boldsymbol{\varphi}}) 
	= \alpha_j \mathcal{J}_{1,B}(\tilde{\vec I}, \tilde{\boldsymbol{\varphi}}) + \frac12 \Vert B - B_j \Vert_{\mathcal Y}^2.
	\end{equation}
	Since the two rightmost terms converge to zero, this implies $A\hat I_0 = B$.
\end{proof}

%-----------------------------------------------------
\section{Minimization Approaches} \label{sec:disc_minimization}		
%-----------------------------------------------------
In this section, we propose two different alternating minimization schemes.
The first one is known as PALM and  updates in each step the deformations and images via proximal computations.
Convergence of the whole iteration sequence to a critical point is ensured.
The second one just alternates the minimization of the deformations and the images.

Note that solving the \emph{coupled} problem in $(\vec I,\boldsymbol{\varphi})$ e.g.~with a gradient scheme is very time and memory consuming, 
since all $I_k$ and $\varphi_k$ are treated at the same time.
Moreover, the coupling $I_{k}^{(j)} \circ \varphi_{k}^{-1}$ is non-convex and hence it is difficult to provide convergence results for general schemes.

We start with the spatial discretization of $\mathcal J$ in \eqref{path_I_phi}.

\subsection{Spatial Discretization}\label{sec:discretization}

\begin{figure}[ht]
\centering
\includegraphics[width = .6\textwidth]{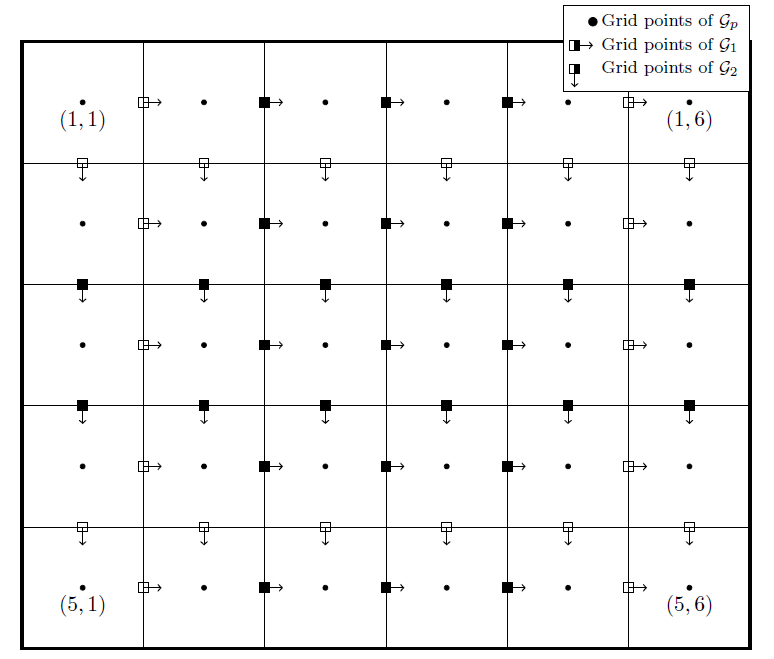}%	
	\caption{Illustration of the staggered grid, where empty boxes mean zero movement.\label{fig:grid}}
\end{figure}
Dealing with rectangular digital images, we propose a finite difference approach, where we work on staggered grids, see Fig.~\ref{fig:grid}.
%For an application of mimetic grid techniques in optical flow computation see also \cite{YSS07}.
In the following, the spatial discretization is briefly sketched.
The domain of the images $I$ is the (primal) grid 
$\grid\coloneqq \{1,\ldots,n_1\} \times \{1,\ldots,n_2\} $. 
All integrals are approximated on the integration domain 
$\overline \Omega\coloneqq [\tfrac12,n_1 + \tfrac12] \times [\tfrac12,n_2 + \tfrac12]$
by the midpoint quadrature rule, i.e., with pixel values defined on $\grid$.
Further, it is assumed that for the operator $A$ a discrete version $A\colon \mathcal G \to Y$ is known, where $Y$ is some finite dimensional Hilbert space. 

First, we discuss the discretization of $\mathcal F$. As regularizer $W(D\varphi)$ we propose the linearized elastic potential $S(v)$ from Remark~\ref{lin_elastic_pot}
with the replacement $v = (v_1,v_2)^\tT = \varphi - \mathrm{id}$.
Using the $\frac12$-shifted grids
\[
\grid_1 \coloneqq \{\tfrac32,\ldots,n_1- \tfrac12\} \times \{1,\ldots,n_2\},\quad
\grid_2 \coloneqq \{1,\ldots,n_1\} \times \{\tfrac12,\ldots,n_2- \tfrac32\},
\]
we consider $v = (v_1,v_2)^\tT$ with $v_1\colon   \grid_1 \rightarrow \mathbb R$ 
and $v_2\colon \grid_2 \rightarrow \mathbb R$.
Then, the spatially discrete version of $S$ reads 
\begin{align*}	
	\mathcal{S}(v) &= \mu \left( \left\|D_{1,x_1} v_1 \right\|_F^2 +  \left\| v_2 D_{2,x_2}^\tT \right\|_F^2 + \frac{1}{2} \bigl\|v_1 D_{1,x_2}^\tT 
	+  D_{2,x_1} v_2 \bigr\|_F^2\right)\\&\quad
	+ \frac{\lambda}{2} \bigl\|D_{1,x_1} v_1 + v_2 D_{2,x_2}^\tT  \bigr\|_F^2,
	\end{align*}
where $\| \cdot \|_F$ is the Frobenius norm of matrices 
and $D_{i,x_j}$ denotes the forward differences operator (matrix) for $v_i$ in $x_j$-direction.
The higher order term $\int_\Omega \lvert D^m\varphi\rvert^2\, \dx x$ 
with $m=3$ is discretized by	
\[
{\mathrm D}_3(v) \coloneqq \sum_{i=0}^{3} \bigl\|D_{1,x_1}^i v_1 D_{1,x_2}^{3-i}  \bigr\|_F^2 + \bigl\|D_{2,x_1}^i v_2 D_{2,x_2}^{3-i}  \bigr\|_F^2 
+ \eta \bigl(\lVert v_1\rVert_F^2+\lVert v_2\rVert_F^2\bigr),
\]
where central differences operators are used for the partial derivatives of order two and three. 
Note that we added the squared Frobenius norm of the $v_i$, $i=1,2$ for a better control of the displacement value.
To cope with the remaining deformation term in \eqref{path}, we approximate $\varphi^{-1} \approx \mathrm{id}-v$
such that the data term simplifies to
\begin{equation}\label{eq:DataTerm}
	\int_\Omega \bigl\vert I_k^{(j)}\bigl(x-v(x)\bigr) - I_{k+1}^{(j)}(x) \bigr\vert^2 \, \dx x.
\end{equation}
This integral is evaluated using the midpoint quadrature rule. 
Since $v_i$ is only defined on $\grid_i$, $i=1,2$ and not on $\grid$, the averaged version $Pv =(P_1v_1,P_2v_2)^\tT \colon \grid \to \RR^2$ is used.	
In general $x - Pv(x) \not \in {\cal G}$, so that the image $I_k\left( x-Pv(x) \right)$ has to be interpolated from its values on $\grid$.
For this purpose linear interpolation with an interpolation matrix $P_I$ is used.
Note that also interpolation matrices with higher space regularity or splines can be used.
Summarizing, the discrete version of \eqref{path} reads
\begin{equation}
	{\mathcal F}(\boldsymbol v; \boldsymbol I) \coloneqq \sum_{k=1}^{K}\mathcal{S}(v_k) + \nu {\mathrm D}_3(v_k) 
	+ \sum_{x \in \grid} \bigl \vert P_I\bigl(x-Pv_k\bigr)I_k - P_I(x)I_{k+1} \bigr\vert^2.
\end{equation}
	
It remains to discretize $\mathcal E$, which is done by using the midpoint rule for the data term.
For the TV-term the forward differences $D_{x_i}$ in $x_i$, $i=1,2$, direction are used
\[
TV(I) \coloneqq \left \| \sqrt{(D_{x_1} I)^2  + (I D_{x_2}^\tT)^2} \right \|_1,
\]
where the square and the square root are meant componentwise, and $\| \cdot \|_1$ is the sum of the entries of the matrix.
Then, the discrete functional reads
\[
\mathcal E(I;B) = \Vert AI - B \Vert_F^2 + \left \| \sqrt{(D_{x_1} I)^2  + (I D_{x_2}^\tT)^2} \right \|_1.
\]

%--------------------------------------------------------------------------------------------
\subsection{PALM}\label{sec:palm}
%--------------------------------------------------------------------------------------------
Our first approach for the minimization of $\mathcal J$ is based on PALM \cite{BST14, PS17}. 
This algorithm aims to minimize a functional
\begin{equation} \label{palm_functional}
	\argmin_{x_1 \in E_1,x_2 \in E_2} \bigl\{ H(x_1,x_2) + G_1(x_1) + G_2(x_2) \bigr\}
	\end{equation}
by iterating 
	\begin{align}  \label{palm}
		&x_1^{(j+1)} = \prox_{\tau G_1}\left(x_1^{(j)} - \frac{1}{\tau} \nabla_{x_1} H\bigl(x_1^{(j)}, x_2^{(j)}\bigr)\right),\\
		&x_2^{(j+1)} = \prox_{\sigma G_2}\left(x_2^{(j)} - \frac{1}{\sigma} \nabla_{x_2} H\bigl(x_1^{(j+1)}, x_2^{(j)}\bigr)\right)	,	
	\end{align}
where $\tau,\sigma > 0$ 
and $\prox_{\tau f}	(x) \coloneqq \argmin_{y} \frac12 \|x-y\|_2^2 + \tau f(y)$ denotes the proximal mapping of $f$,
which is uniquely determined for proper, convex and lsc functions $f$.
The convergence result is stated in the following theorem from \cite[Theorem~1]{BST14}. 
Note that in the theorem the proximal map is also defined for non-convex functions. 
However, the involved functions in our application are convex, so that no further details on this topic are provided.
%	Note that both $f$ and $g$ are proper and lsc and that $H$ is differentiable (with Lipschitz gradient?). In this context the steps of iPALM read as follows (for simplicity i omitted the inertial steps)
%	\begin{align}
%		&\vec I^{(j+1)} = \prox_{\tau f}\left(\vec I^{(j)} - \frac{1}{\tau} \nabla_{\vec I} H(\vec I^{(j)}, \vec v^{(j)})\right)\\
%		&\vec v^{(j+1)} = \vec v^{(j)} - \frac{1}{\sigma} \nabla_{\vec v} H(\vec I^{(j+1)}, \vec v^{(j)}).
%	\end{align}
%	Here the first proximal map can be computed using a primal dual method, and the second step can be computed separately for all $v_k$. 
%\end{remark}
%------------------------------------
\begin{theorem}\label{thm:palm}
	Let $E_1, E_2$ be Euclidean spaces and $H\colon E_1 \times E_2 \rightarrow \mathbb R \cup\{+\infty\}$, 
	$G_i\colon E_i \rightarrow \mathbb R \cup\{+\infty\}$, $i=1,2$, 
	be proper, lsc functions.
	Assume that $H$ is continuously differentiable with locally Lipschitz continuous gradient 
	and that both $x_i \mapsto \nabla_{x_i} H(x_1,x_2)$ are globally Lipschitz, 
	where the constants $L_1(x_2), L_2(x_1)$ possibly depend on the fixed variable.
	Let $G_1 + G_2 + H$ in \eqref{palm_functional} fulfill the Kurdyka–Łojasiewicz (KL) property.
	Further, assume $\tau > L_1(x_2^{(j)})$ and $\sigma > L_2(x_1^{(j)})$ for all $j \in \NN$.
	If the sequence generated by \eqref{palm} is bounded, then it converges to a critical point.
\end{theorem}
%------------------------------------
For our problem we choose the splitting
\begin{align}
	&G_1(\vec I) =  \alpha TV(I_0) + \frac{1}{2} \Vert A I_0 - B\Vert_F^2,\\
	&G_2(\vec v) =  0,\\
	&H(\vec I, \vec v) = \beta \left( \sum_{k=0}^{K-1} \sum_{x \in \grid} \bigl|P_I\bigl(x-P_v v_k(x)\bigr)I_{k} - P_I(x)I_{k+1}\bigr|^2 + \alpha \mathcal{S}(v_k) + \nu {\mathrm D}_3(v_k)\right).  
\end{align}
Then the iteration \eqref{palm} reads
\begin{align}
	&\vec I^{(j+1)} = \prox_{\tau G_1}\left(\vec I^{(j)} - \frac{1}{\tau} \nabla_{\vec I} H\bigl(\vec I^{(j)}, \vec v^{(j)}\bigr)\right),\\
	&\vec v^{(j+1)} = \vec v^{(j)} - \frac{1}{\sigma} \nabla_{\vec v} H\bigl(\vec I^{(j+1)}, \vec v^{(j)}\bigr).
\end{align}
From the structure of $H$ we deduce that the $v_k$, $k=0,\ldots,K-1$, can be computed separately.
Note that the second term in $G_1$ can also be added to $H$, but this only makes sense if $\Vert A^\tT A \Vert$ is small since otherwise the Lipschitz constant gets too large.
The $I_k$, $k=1,\ldots,K$, can also be updated separately \cite[Section 3.6]{BST14}, which possibly improves the Lipschitz constants.
If the interpolation matrix $P_I$ originates from smooth piecewise polynomial basis functions and $\tau, \sigma$ are chosen accordingly, all conditions of Theorem~\ref{thm:palm} are satisfied.
The proximal map $\prox_{\tau G_1}$ can be computed efficiently by primal-dual algorithms from convex analysis
as for  example by the Chambolle-Pock algorithm \cite{CP11,PCCB09}.
Indeed there is a vast literature how to solve problems of this kind, see e.g.~\cite{BSS2014,CP16} for an overview.

\subsection{Alternating Minimization Approach}\label{sec:alternativapproach}
%----------------------------------------------------------------------------------
The computation of $\prox_{\tau G_1}$ with a primal-dual algorithm requires an inner iteration for every step of PALM.
If the evaluation of $A$ is computationally expensive, this can result in high computation effort due to many operator evaluations.
Our numerical experiments indicated that PALM needs relatively many outer iterations and hence also many evaluations of the operator.
Therefore, we want to present a second alternating scheme to minimize $\mathcal{J}(\vec I,\boldsymbol{\varphi})$ 
which needed fewer outer iterations in our experiments.
%Unfortunately, we are not able to give a proper convergence analysis for this scheme. 

Starting with $\vec I^{(0)},\vec \varphi^{(0)}$ we iterate for $j=0,\ldots$:
\\
1. For  $k= 0,1,\ldots,K-1$, we compute
\begin{equation} \label{deform}
\varphi_k^{(j)} = \argmin_{\varphi_k \in \mathcal A}
\left\{ 
\int_\Omega W(D\varphi_k) + \nu\lvert D^m\varphi_k\rvert^2
+ \bigl| I_{k}^{(j)} \circ \varphi_{k}^{-1}  - I_{k+1}^{(j)}\bigr|^2 \dx x \right\}. 
\end{equation}
2. For given $A \in L(L^2(\Omega),{\mathcal Y})$, $B \in {\mathcal Y}$ and $R \in L^2(\Omega)$, we solve
\begin{align}
	\begin{split}
	\vec I^{(j)} 
	=& 
	\argmin_{\vec I \in (L^2(\Omega))^K} \left\{ 
	\beta \sum_{k=0}^{K-1} \bigl\Vert I_{k} \circ ( \varphi_{k}^{(j)})^{-1}  - I_{k+1} \bigr\Vert_{L^2(\Omega)}^2 
	+ \frac{1}{2} \| A I_0 - B\|^2_{\mathcal Y} + \alpha TV(I_0) \right\}.\label{im_seq}
	\end{split}
\end{align}
For the first step the discretization from Section~\ref{sec:discretization} is applied which results in the minimization of
\begin{equation} \label{eq:regu}
{\mathcal R}\bigl(v_k;I_k^{(j)}, I_{k+1}^{(j)}\bigr) \coloneqq \mathcal{S}(v_k) + \nu {\mathrm D}_3(v_k) 
+ \sum_{x \in \grid} \bigl \vert P_I\bigl(x-Pv_k\bigr)I_k - P_I(x)I_{k+1} \bigr\vert^2,
\end{equation}
for $k=0,\ldots,K-1$.
This problem can be solved by a Quasi-Newton method, details can be found in \cite{NPS17,PPS17}.

For the computation of the image sequence in the second step of the algorithm we use the substitution from the proof of Lemma~\ref{lem:uni:seq}.
Setting 
$\psi_k \coloneqq \varphi_{k-1} \circ \ldots \circ \varphi_0$, 
$w_k(x) \coloneqq  \det\left( D \psi_k(x) \right)$ and
$F_0 \coloneqq I_0$, $F_k \coloneqq I_k \circ \psi_{k}$, 
we can transform \eqref{im_seq} to
\begin{equation}\label{CoupledReco}
\argmin_{\vec F} \left\{ 
\beta \sum_{k=0}^{K-1} \| (F_{k} - F_{k+1}) \sqrt{w_{k+1}}\|_{L^2(\Omega)}^2 
\; + \frac{1}{2} \| A F_0 - B\|^2_{\mathcal Y} + \alpha TV(F_0)\right\}.  
\end{equation}
The functional is discretized on $\grid$, using the approach from Section~\ref{sec:discretization}. 
We propose to solve the discrete version of \eqref{CoupledReco} with a block-coordinate descent
which fixes alternately $F_0$ and $\bar {\vec F} \coloneqq (F_1,\ldots,F_{K-1})$.
For block-coordinate descent the following convergence result was proven in~\cite[Theorem~14.9, Theorem~14.15]{Beck2017}, see also~\cite{BT13}.
%------------------------------------
\begin{theorem}\label{thm:alt_min}
	Let $E_1, E_2$ be Euclidean spaces and $G\colon E_1 \times E_2 \rightarrow \mathbb R \cup\{+\infty\}$, $G_i\colon E_i \rightarrow \mathbb R \cup\{+\infty\}$, $i=1,2$ 
	be proper, convex lsc functions.
	Assume further that $G$ is continuously differentiable and that the level sets of $G + G_1 + G_2$ are bounded. Then the minimization problem
	\begin{equation*}
	\argmin_{x_1 \in E_1,x_2 \in E_2} \bigl\{ G(x_1,x_2) + G_1(x_1) + G_2(x_2) \bigr\}
	\end{equation*}
	can be solved by alternating minimization in $x_1$ and $x_2$, 
	i.e., every accumulation point of the generated iteration sequence is a minimizer. 
	The convergence rate for the functional values is $\mathcal{O}(\frac{1}{k})$.
\end{theorem}
%------------------------------------
For our specific discretized problem \eqref{CoupledReco} with
\begin{align*}
G(F_0,\bar {\vec F})  &\coloneqq \beta \sum_{x\in \grid} | F_0(x)- F_1(x) |^2 w_{1}(x),\\
G_1(\bar {\vec F} )&\coloneqq \beta \sum_{k=0}^{K-1} \sum_{x \in \grid} | F_{k}(x) - F_{k+1}(x)|^2 w_{k+1}(x),\\
G_2(F_0) &\coloneqq \frac{1}{2} \| A F_0 - B\|^2_{\mathcal Y} 
+ \alpha TV(F_0),
\end{align*}	
the conditions of the theorem are obviously fulfilled. 
If $F_0$ is fixed, Corollary \ref{cor:tridiag} implies that the minimizer of $G(F_0,\bar {\vec F}) + G_1(\bar {\vec F})$
is given analytically.
In the second step of the algorithm we have to minimize, for fixed $\bar{\vec F}$,
the functional
\begin{equation}\label{eq:OurReco}
G(F_0,\bar {\vec F}) +	G_2(F_0) = \beta \sum_{x \in \grid} |F_{0}(x)  - F_1(x)|^2 w_{1}(x) 
+ \frac{1}{2} \| A F_0 - B\|^2 
+ \alpha TV(F_0).
\end{equation}
This can be done efficiently by primal-dual algorithms from convex analysis, see Section~\ref{sec:palm} for a discussion.
Finally, we use scattered interpolation to obtain the images $\vec I$ at grid points from $\vec F$.
%-----------------------------------------	
\subsection{Multilevel Approach}
%-----------------------------------------
	As usual in optical flow and image registration, we apply a coarse-to-fine strategy with $\operatorname{lev}\in\NN$ levels if a downsampling procedure for the data and the operator is known.
	This is the case for our numerical experiments, but it is also possible to use only a single level if no downsmapling procedure is known.
	First, we iteratively smooth our given template image by convolution with a truncated Gaussian and downsampling using bilinear interpolation. 
	Here special care is necessary for the operator $A$, 
	as well as for the downsampling procedure of the data $B$, 
	which is dependent on the operator choice. 
	Both procedures are described in the respective numerical examples.
	
	In order to obtain a deformation on the coarsest level, a single registration is performed with the solution of the $L^2$-$\TV$ problem, i.e.,
	\begin{equation}\label{TVReco}
	I_{0,\operatorname{lev}} = 
	\argmin_{ I\colon\grid_{\operatorname{lev}}\to\RR}\left\{ \frac12 \lVert A_{\operatorname{lev}} I - B_{\operatorname{lev}}\rVert^2
	+\ \alpha \TV(I)\right\},\quad \alpha > 0,
	\end{equation}
	where $\operatorname{lev}\in\NN$ is number of levels. 
	For better results the regularization parameters for $v$ are decreased successively as recommended by Modersitzki \cite{Mod2009}.
	
	After computing a solution on every level, bilinear interpolation is applied to construct an initial deformation on the next finer level. 
	The sequence of $\tilde K -1$, $\tilde K < K$, intermediate finer level images is initialized from the end
	\begin{equation}\label{eq:inter:img}
	I_k(x) = R\bigl(x+\tfrac{k}{\tilde K}  Pv(x)\bigr),
	\end{equation}
	where $R$ is the template image at the current level. 
	Using this we obtain an initial image sequence on this level.
	The complete multilevel strategy is sketched in Algorithm~\ref{alg:morph} for the alternating minimization scheme presented in Section~\ref{sec:alternativapproach}.
	
	%----------------------------------------------------
	\begin{algorithm}[htb]
		\caption{TDM-INV Algorithm (informal) }\label{alg:morph}
		\begin{algorithmic}[1]
			\State $R_0 \coloneqq  R, B_0 \coloneqq B, \grid_0 \coloneqq \grid$
			\State create image stack $(R_l)_{l=0}^{\operatorname{lev}}, (B_l)_{l=0}^{\operatorname{lev}}$ on $(\grid_l)_{l=0}^{\operatorname{lev}}$ 
			by downsampling
			\State solve \eqref{TVReco} for $B_{\operatorname{lev}}$
			\State solve \eqref{eq:regu} for $R_{\operatorname{lev}}, I_{0,\operatorname{lev}}$ to get $\tilde{v}$
			\State $l \to \operatorname{lev}-1$
			\State use bilinear interpolation to get $v$ on $\grid_{l}$ from $\tilde{v}$
			\State obtain $\tilde{K}_{l}$ images $\vec I_{l}^{(0)}$ from $R_{l},v$ by \eqref{eq:inter:img}
			\While{$l\ge0$}
			\Repeat (Alternating outer iteration)
			\State find deformations $\tilde{\vec v}_l^{(i+1)}$ minimizing \eqref{eq:regu} for every pair from $\vec I_l^{(i)}$
			\State initialize $\vec F^{(0)} = \vec I_l^{(i)}$
			\Repeat (Alternating inner iteration)
			\State for fixed $F_0^{(j)}$ compute $F_1^{(j+1)},\cdots, F_{K-1}^{(j+1)}$ according to Corollary~\ref{cor:tridiag}
			\State for fixed $F_1^{(j+1)}$ compute $F_0^{(j+1)}$ as solution of \eqref{eq:OurReco} using a PD-method
			\State $j \to j+1$
			\Until convergence criterion is reached
			\State compute $\vec I_l^{(i)}$ from $F^{(j)}$ using scattered interpolation
			\State $i \to i+1$
			\Until convergence criterion is reached
			\State $l\to l-1$
			\If{$l>0$}
			\State use bilinear interpolation to get $\vec I_l$ and $\vec  v_l$ on $\grid_l$
			\For{$k = 1,\dots,\tilde K _l$}
			\State calculate $\tilde{K}_l$ intermediate images between $I_{l,k-1},I_{l,k}$ with $v_{l,k}$ using~\eqref{eq:inter:img}
			\EndFor
			\EndIf
			\EndWhile	
			\vspace{0.4cm}\State $\vec I \coloneqq \vec I_0$
		\end{algorithmic}
	\end{algorithm}

%-----------------------------------------------------
\section{Numerical Examples} \label{sec:numerics}		
%-----------------------------------------------------
In this section, numerical examples demonstrating the potential of the method are presented.
The proposed Algorithm~\ref{alg:morph} is implemented using Matlab.
We also implemented the minimization of the TDM-INV model using PALM, but observed higher computation times due to many operator evaluations.
For the Radon transform, the computation roughly needed two times as long (about 5-10 minutes).
As comparison a result using PALM is added in the first example.
The qualitative differences between the two results are very small 
and therefore only the results of Algorithm~\ref{alg:morph} are shown in the remaining experiments.
Note that PALM might be more favourable if the operator $A$ is simple to evaluate, e.g.~if it is sparse.

For representing our images on a grid during the registration step,
we applied the \lstinline|mex| interface of the spline library by E.~Bertolazzi~\cite{Splines} 
with the Akima splines.
In order to reduce the number of involved parameters in \eqref{eq:regu}, 
we use $\lambda = \mu = \nu = 100\eta$ in all our experiments.
Typical choices for the increments $\tilde K$ are $\tilde K_{\text{lev}-1} = 2$, $\tilde K_{\text{lev}-2}=1$ and $\tilde K_i=0$ for the remaining levels.
The remaining parameters $\alpha$, $\beta$ and $\lambda$ are optimized with respect to the SSIM via a gridsearch.
For the comparison algorithms the parameters are SIMM optimized, too.
A GPU implementation is applied for solving the appearing linear systems of equations in the Quasi-Newton method.

In the first part of our experiments, the Radon transform is considered as operator.
Among the vast literature on the topic, 
we refer to the books \cite{He1980,Ku2014,N2001}
for a general introduction to CT
including some reconstruction methods from incomplete data
and for limited angle tomography e.g.~to \cite{Dav1983,HD2008,Louis1986}.
The second part deals with superresolution, 
which does not have a continuous counterpart.

%----------------------------------------------------
\subsection{Limited Angle and Sparse CT}
%----------------------------------------------------
We are given a reference image $R \in [0,1]^{256,256}$ 
%on the grid $\grid_0 = \{1,\dots,256\}\times\{1,\dots,256\}$ 
and sinogram data of a target image $I_{\mathrm{orig}} \in[0,1]^{256,256}$, 
which we want to reconstruct. 
For the numerical implementation of the (discrete) Radon transform the Astra toolbox \cite{AstraGPU,Astra2,Astra1} is used, 
which allows more flexibility compared to the built-in Matlab function.

In our \textbf{first example}, the reference image consists of 6 triangular shaped objects, which are deformed to stars in the target image, see Fig.~\ref{fig:triangles}\footnote{The images in Fig.~\ref{fig:triangles}(a) and (b) are taken from the paper \cite{CO18} 
and were provided by Barbara Gris and Ozan \"Oktem.}. 
The sinogram is obtained by the Radon transform
using 10 measurement directions equally distributed (with steps of 9 degrees) from 0 to 81 degrees, 
i.e.~the measurement angle is limited to less than the half domain. 
The sinogram is additionally corrupted with 5 percent Gaussian noise. 
Our goal is to reconstruct the target from the given sinogram data. 
In the proposed multi grid approach a down-sampling by a factor of 0.5 is used. 
For the down-sampling of the sinogram, two neighboring rays are averaged and rescaled to the correct intensity. 
Note that this is easily possible if the number of rays is chosen for example to be 1.5 times the number of pixels per direction.
The result of our TDM-INV algorithm is shown in Fig.~\ref{fig:triangles:ourresult},
where the parameters $\operatorname{lev}=4$, $\lambda=0.07$, $\alpha = 0.05$ and $\beta = 0.1$ are used. 
Compared to the reconstruction by the $L^2$-$TV$ model (with $\lambda_{\TV} = 0.05$) in Fig.~\ref{fig:triangles:L2TV}, our method is able to better deal with the missing data from 81 to 180 degrees.
Visually, the result is almost perfect and also the SSIM value is very good.
In Fig.~\ref{fig:triangles:palm} the numerical result using PALM is shown. 
The SSIM and PSNR values are similar to Fig.~\ref{fig:triangles:ourresult} 
and almost no difference is visible. 
The difference of both results is depicted in Fig.~\ref{fig:triangles:diff} and lies within the color range $[-0.07, 0.06]$.

%-----------------------------------
\begin{figure}
	\centering
	\begin{subfigure}[t]{0.35\textwidth}
		\centering
		\includegraphics[width = 0.98\textwidth]{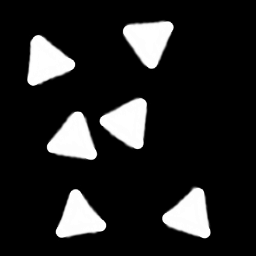}
		\caption{Reference image.}
	\end{subfigure}		
	\begin{subfigure}[t]{0.35\textwidth}
		\centering
		\includegraphics[width = 0.98\textwidth]{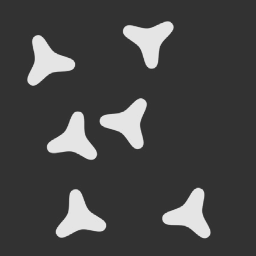}
		\caption{Target image.}
	\end{subfigure}\\	
	\begin{subfigure}[t]{0.35\textwidth}
		\centering
		\includegraphics[width = 0.98\textwidth]{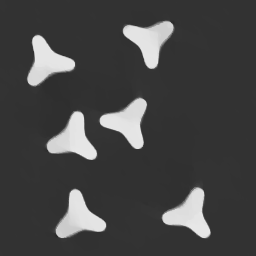}
		\caption{Result by TDM-INV.\\ (SSIM .9815, PSNR 30.31)}\label{fig:triangles:ourresult}
	\end{subfigure}
	\begin{subfigure}[t]{0.35\textwidth}
		\centering
		\includegraphics[width = 0.98\textwidth]{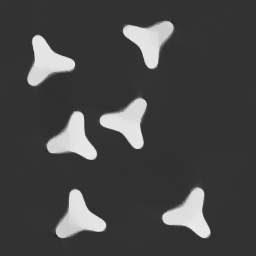}
		\caption{Result by TDM-INV using PALM.\\ (SSIM .9815, PSNR 30.40)}\label{fig:triangles:palm}
	\end{subfigure}
	\begin{subfigure}[t]{0.35\textwidth}
		\centering
		\includegraphics[width = 0.98\textwidth]{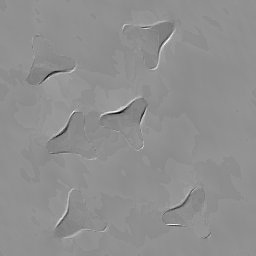}
		\caption{Difference of the results with range $[-0.07, 0.06]$.}\label{fig:triangles:diff}
	\end{subfigure}		
	\begin{subfigure}[t]{0.35\textwidth}
		\centering
		\includegraphics[width = 0.98\textwidth]{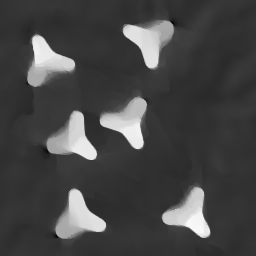}
		\caption{Result by $L^2$-$TV$.\\ (SSIM .9377, PSNR 24.89)}\label{fig:triangles:L2TV}
	\end{subfigure}
	\caption{Image reconstruction from sparse, limited angle CT measurements from 0 to 81 degrees with 10 angles.}\label{fig:triangles}	
\end{figure}
%----------------------------------------------------

In the \textbf{second example} a more structured image is treated. 
The given reference image depicts an artificial brain image, and the target can be considered as a deformed version, 
see Fig.~\ref{fig:tomo}\footnote{Available at \url{http://bigwww.epfl.ch/algorithms/mriphantom/}, see also~\cite{GLPU12}.}. 
The sinogram of the target is created using the Radon transform with 
20 measurements equally distributed from 0 to 180 degrees and by adding 5 percent Gaussian noise.

%------------------------------------------------------------------------------------
\begin{figure}
	\centering
	\begin{subfigure}[t]{0.45\textwidth}
		\centering
		\includegraphics[width = 0.98\textwidth]{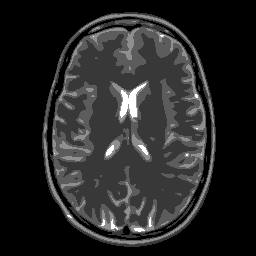}
		\caption{Reference image.}
	\end{subfigure}		
	\begin{subfigure}[t]{0.45\textwidth}
		\centering
		\includegraphics[width = 0.98\textwidth]{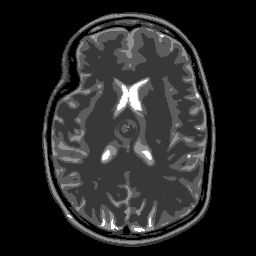}
		\caption{Target image.}
	\end{subfigure}		
	\begin{subfigure}[t]{0.45\textwidth}
		\centering
		\includegraphics[width = 0.98\textwidth]{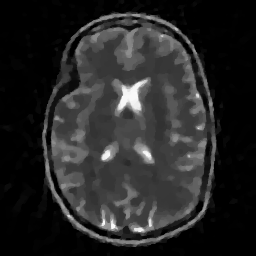}
		\caption{Result by TDM-INV. \\(SSIM .7542, PSNR 26.47)}\label{fig:tomo:ourresult}
	\end{subfigure}		
	\begin{subfigure}[t]{0.45\textwidth}
		\centering
		\includegraphics[width = 0.98\textwidth]{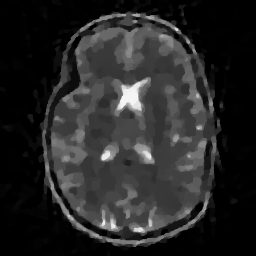}
		\caption{Result  by $L^2$-$TV$\label{fig:tomo:L2TV}.\\(SSIM .6819, PSNR 24.20)}
	\end{subfigure}
	\caption{Image reconstruction from sparse CT measurements using 20 angles from 0 to 180 degrees.}\label{fig:tomo}
\end{figure}
%-------------------------------------------------------------------------------------------------------------
For the multi grid approach the procedure from the previous example is used.
The result of TDM-INV is shown in Fig.~\ref{fig:tomo:ourresult} and was calculated 
with the parameters $\operatorname{lev}=5$, $\lambda=0.08$, $\alpha = 0.025$ and $\beta = 0.5$.
Since our model incorporates the reference information as compensation for the sparse data set, the reconstruction is better than the one with the $L^2$-$TV$ model (with $\lambda_{\TV} = 0.1$) in Fig.~\ref{fig:tomo:L2TV}.

%----------------------------------------------------------------------
\subsection{Superresolution}
%----------------------------------------------------------------------
\begin{figure}
	\centering
	\begin{subfigure}[t]{0.32\textwidth}
		\centering
		\includegraphics[width = 0.98\textwidth]{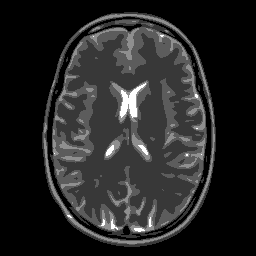}
		\caption{Reference image.}
	\end{subfigure}		
	\begin{subfigure}[t]{0.32\textwidth}
		\centering
		\includegraphics[width = 0.98\textwidth]{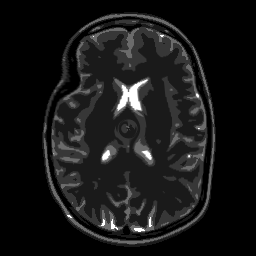}
		\caption{Target image.}
	\end{subfigure}	
	\begin{subfigure}[t]{0.32\textwidth}
		\centering
		\includegraphics[width = 0.98\textwidth]{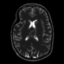}
		\caption{Low resolution image. (SSIM  .7681, PSNR  24.02)}
	\end{subfigure}	
	
	\begin{subfigure}[t]{0.32\textwidth}
		\centering
		\includegraphics[width = 0.98\textwidth]{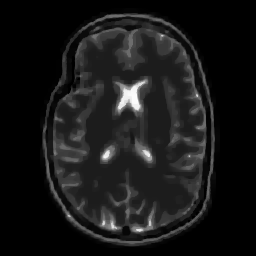}
		\caption{Result by TDM-INV. (SSIM  .8767, PSNR~27.46)}\label{fig:sr:ourresult}
	\end{subfigure}		
	\begin{subfigure}[t]{0.32\textwidth}
		\centering
		\includegraphics[width = 0.98\textwidth]{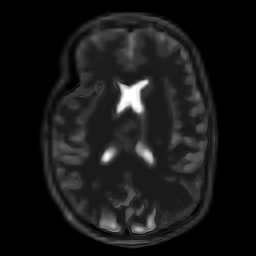}
		\caption{Result by He/Siu~\cite{HS11}. (SSIM  .7823, PSNR  24.25)}\label{fig:sr:gpr}
	\end{subfigure}
			
	\begin{subfigure}[t]{0.32\textwidth}
		\centering
		\includegraphics[width = 0.98\textwidth]{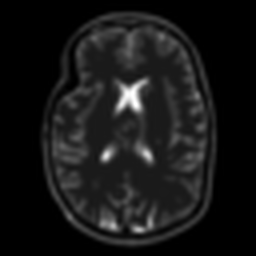}
		\caption{Result by Matlab SR. \\(SSIM  .8111, PSNR  25.76)} \label{fig:sr:matlab}
	\end{subfigure}		
	\begin{subfigure}[t]{0.32\textwidth}
		\centering
		\includegraphics[width = 0.98\textwidth]{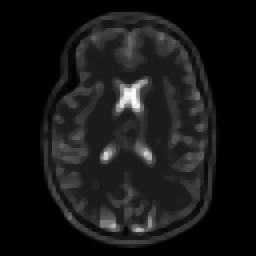}
		\caption{Result by $L^2$-$TV$.\\
		(SSIM  .8075, PSNR 24.75)}\label{fig:sr:l2tv}
	\end{subfigure}
	\caption{Superresolution from $64\times64$ pixels to $256\times256$ for brain image.}\label{fig:sr}
\end{figure}
%---------------------------------------------

Here, we are given a reference image $R \in [0,1]^{256,256}$ 
and a low resolution image $B\in[0,1]^{64,64}$ 
obtained by down-sampling of a target image $I_{\mathrm{orig}}\in [0,1]^{256,256}$
with the down-sampling operator $P_4\in\RR^{256,64}$ given by
\begin{equation*}
P_{4} = \frac{1}{4}	\left(\begin{array}{@{}*{12}{c}@{}}		
1 & 1 & 1 & 1 & 0 &   &   &   &   &  &  & 0\\
0 & 0 & 0 & 0 & 1 & 1 & 1 & 1 & 0 &  & & 0\\
&   &   &   &\ddots&&   &\ddots&& &  &  \\
0 &   &   &   &   & &  & 0 & 1 & 1 & 1 & 1 
\end{array}\right)\in\RR^{256,64}.
\end{equation*} 
In other words, $B = P_4 I_{\mathrm{orig}} P_4^\tT$.
For the multi grid approach a downscaling with factor $0.5$ is applied such that the given image $B$ can be used for the first three levels, i.e., $B_0=B_1=B_2\in \RR^{64,64}$. 
The matrix $P_4\in\RR^{256,64}$ is adapted to $P_2\in\RR^{128,64}$ for the second level and the identity matrix of corresponding size is used for all higher levels. 

In our \textbf{third example} the same reference and target images as in the second example are used, see Fig.~\ref{fig:sr}.
The result of TDM-INV is shown in Fig.~\ref{fig:sr:ourresult}, where the parameters $\operatorname{lev}=4$, $\lambda = 0.01$, $\alpha = 0.001$ and $\beta = 2$ are used.
First, our method is compared with the single image superresolution method 
of He and Siu~\cite{HS11}, which is based on a self-similarity assumption of the high and low resolution image together with a Gaussian process regression. 
In contrast to the result obtained by this method in Fig.~\ref{fig:sr:gpr}, 
our result does not have artifacts around the bright features. 
Using the Matlab function \lstinline|imresize|, the best reconstruction is obtained with the ``lanczos3'' kernel, see Fig.~\ref{fig:sr:matlab}, 
which is affected by a strong blur.
For this example, the $L^2$-$\TV$  (parameter $\lambda_{\TV} = 0.001$)
reconstruction yields the result shown in Fig.~\ref{fig:sr:l2tv}. 
Comparing all methods, we see that our method is best at recovering the fine details as well as the overall structure.

\begin{figure}
	\centering
	\begin{subfigure}[t]{0.32\textwidth}
		\centering
		\includegraphics[width = 0.98\textwidth]{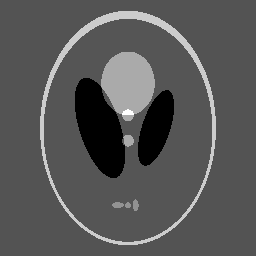}
		\caption{Reference image.}\label{fig:sr2:source}
	\end{subfigure}		
	\begin{subfigure}[t]{0.32\textwidth}
		\centering
		\includegraphics[width = 0.98\textwidth]{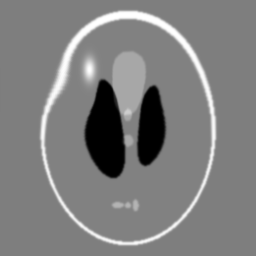}
		\caption{Target image.}\label{fig:sr2:target}
	\end{subfigure}	
	\begin{subfigure}[t]{0.32\textwidth}
		\centering
		\includegraphics[width = 0.98\textwidth]{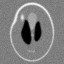}
		\caption{Low resolution image. (SSIM  .7284, PSNR 25.47)}\label{fig:sr2:given}
	\end{subfigure}	
	
	\begin{subfigure}[t]{0.32\textwidth}
		\centering
		\includegraphics[width = 0.98\textwidth]{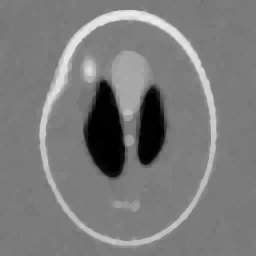}
		\caption{Result by TDM-INV. (SSIM .9345, PSNR~28.98)}\label{fig:sr2:ourresult}
	\end{subfigure}		
	\begin{subfigure}[t]{0.32\textwidth}
		\centering
		\includegraphics[width = 0.98\textwidth]{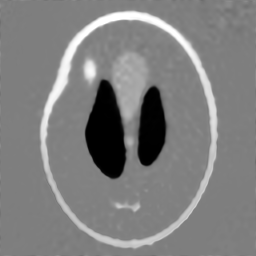}
		\caption{Result by He/Siu~\cite{HS11}. (SSIM .9288, PSNR  28.00)}\label{fig:sr2:gpr}
	\end{subfigure}
	
	\begin{subfigure}[t]{0.32\textwidth}
		\centering
		\includegraphics[width = 0.98\textwidth]{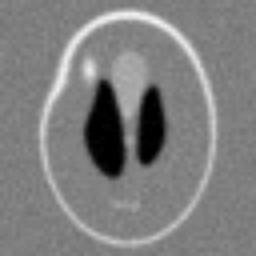}
		\caption{Result by bilinear interpolation (Matlab).\\ (SSIM .8474, PSNR  26.38)}\label{fig:sr2:matlab}
	\end{subfigure}		
	\begin{subfigure}[t]{0.32\textwidth}
		\centering
		\includegraphics[width = 0.98\textwidth]{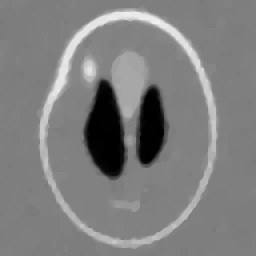}
		\caption{Result by $L^2$-$ \TV$.\\ (SSIM .9226, PSNR 27.49)}\label{fig:sr2:l2tv}
	\end{subfigure}
	\caption{Superresolution from $64\times64$ pixels to $256\times256$ for Shepp-Logan phantom.}\label{fig:sheppdown}
\end{figure}

In our \textbf{last example} , the template image is not only deformed and scaled, 
but also a new detail is included in the image.  
As mass can be created on the image path, our method is able to reconstruct
also the small detail, cf.~Fig.~\ref{fig:sheppdown} \footnote{The images used in Fig.~\ref{fig:sheppdown} are based on the ones in \cite{GO18}.}, where the parameters are chosen as $\operatorname{lev}=4$, $\lambda = 0.01$, $\alpha = 0.001$ and $\beta = 2$.
For this simpler image, our method leads to the best result in SSIM and PSNR. 
The result produced by~\cite{HS11} in Fig.~\ref{fig:sr2:gpr} yields almost the same SSIM , 
but visually the method recovers a lot of background noise. 
The best result of Matlab's \lstinline|imresize| is given by the ``bilinear'' interpolation here.
However, this result is affected by a strong blur. 
The $L^2$-$TV$ approach  (parameter $\lambda_{\TV} = 0.001$) works better for this simpler image than in the previous example, 
but is still not able to match our result. 
Especially the overlapping part in the center of the phantom is only recovered by TDM-INV.

%----------------------------------------------------------------------------------------	
\section{Conclusions} \label{sec:conclusions}
%----------------------------------------------------------------------------------------	
This paper merges the edge-preserving $L^2$-$\TV$ variational mo\-del for solving inverse image reconstruction problems  
with a metamorphosis-inspired approach 
to utilize information from a reference image. 
The approach, called TDM-INV, can handle intensity changes between the reference image and the target image which we want to reconstruct.
The method gives very good results for artificial images so that 
we are looking forward to real-world applications in material sciences or medical imaging, e.g.~motion models for organs \cite{EL13,GJDS15}.
Several extensions of the model are possible. 
Due to the finite difference approach and the design of the method more sophisticated regularizers than the TV-term can be simply involved.
Another possible modification would be to apply different transport models, see e.g.~\cite{MRSS15}.
Further, the usage of multiple reference images can be taken into account. %and the combination with neural networks. 

%--------------------------------------------------------------
\appendix
\section{Gagliardo-Nirenberg Inequality}\label{app:gn}

	\begin{theorem}[Gagliardo-Nirenberg \cite{Nir1966}]\label{th:gn}
		Let $\Omega\subset\RR^n$ be a bounded domain satisfying the cone property. 
		For $1\le q, r\le\infty$, 
		suppose that $f$ belongs to $L^q(\Omega)$ and its derivatives of order $m$ to $L^r(\Omega)$. 
		Then for the derivatives $D^j f$, $0\le j < m$, the following inequalities hold true
		with constants $C_1,C_2$ independent of $f$:
		\begin{equation*}
			\lVert D^j f \rVert_{L^p(\Omega)}\le C_1\lVert D^m f \rVert_{L^r(\Omega)}^a\lVert f\rVert_{L^q(\Omega)}^{1-a}+C_2\lVert f \rVert_{L^q(\Omega)},
		\end{equation*}
		where
		$
			\frac{1}{p} = \frac{j}{n}+a\Bigl(\frac{1}{r}-\frac{m}{n}\Bigr)+(1-a)\frac{1}{q}
		$
		for all $a \in [\frac{j}{m}, 1]$, except for the case $1<r<\infty$ and $m-j-\frac{n}{r}$ is a nonnegative integer, in which the inequality 
		only holds true for  $a \in [\frac{j}{m},1)$.
	\end{theorem}
	%-----------------------------------------------------
	\begin{remark}\label{th:gnr}
		For $p=q=r=2$ the inequality simplifies to
		\begin{align*}
			\lVert D^j f \rVert_{L^2(\Omega)} &\le C_1\lVert D^m f \rVert_{L^2(\Omega)}^{\frac{j}{m}}\lVert f\rVert_{L^2(\Omega)}^{1-\frac{j}{m}}+C_2\lVert f \rVert_{L^2(\Omega)}
			\\
			&\leq C_1\lVert D^m f \rVert_{L^2(\Omega)} + \big(C_1 + C_2 \big) \lVert f \rVert_{L^2(\Omega)},
		\end{align*}
		where the second inequality follows by estimating the product with the maximum of both factors.
	\end{remark}
	
%-------------------------------------------------------------
\subsection*{Acknowledgments} 
    This work was initialized during an internship of S.~Neumayer in the research group of C.~Sch\"onlieb at the University of Cambridge. 
    S.~Neumayer wants to thank B.~Gris, O.~\"Oktem and C.~Sch\"onlieb for stimulating talks on the topic.
    Further, we want to thank A.~Effland for discussions on PALM.
	Funding by the German Research Foundation (DFG) with\-in the project STE 571/13-1 
	and with\-in the Research Training Group 1932,
	project area P3, is gratefully acknowledged.
	We gratefully acknowledge the support of NVIDIA Corporation with the donation of the Quadro M5000 GPU used for this research.

\bibliographystyle{abbrv}
\bibliography{references}	
\end{document}